\documentclass{article}

\usepackage[a4paper,top=2.54cm,bottom=2.54cm,left=3.17cm,right=3.17cm,%
            includehead,includefoot]{geometry}

\usepackage{amsmath,amssymb,amsfonts,amsthm}
\usepackage{mathrsfs}
\usepackage{graphicx}
\usepackage{subfigure}
\usepackage{float}
\usepackage{xcolor}
\usepackage[numbers,square,sort&compress]{natbib}
\usepackage{hyperref}
  \hypersetup{colorlinks,citecolor=blue,linkcolor=blue,breaklinks=true}
\usepackage{booktabs}
\usepackage{colortbl}
\usepackage{caption}
\usepackage{enumitem}
\usepackage{amsmath}
\numberwithin{equation}{section}
\usepackage{epstopdf}
\usepackage{algorithm}
\usepackage{algpseudocode}
\usepackage{array}
\usepackage{bbding}
\usepackage{placeins}
\usepackage{fancyhdr} 
\usepackage{fancyvrb} 
\usepackage{longtable}
\usepackage{listings}
\usepackage{rotating,rotfloat} 
\usepackage{yhmath} 
\newtheorem{conjecture}{Conjecture}

\usepackage{fancyhdr}
\allowdisplaybreaks

\newtheorem{theorem}{Theorem}[section]
\newtheorem{lemma}{Lemma}[section]
\newtheorem{corollary}[theorem]{Corollary}
\newtheorem{proposition}[theorem]{Proposition}
\newtheorem{definition}{Definition}[section]
\newtheorem{remark}{Remark}[section]

\newcommand{\bbm}{\begin{bmatrix}}
\newcommand{\ebm}{\end{bmatrix}}

\begin{document}

	\title{Well-posedness and asymptotic behavior of solutions to a second order nonlocal parabolic MEMS equation}

 \author{Yufei Wei\thanks{School of Mathematical Sciences, East China Normal University, Shanghai 200241,   P.R. China. Email: \texttt{51255500047@stu.ecnu.edu.cn}. },\,
 Yanyan Zhang\thanks{Corresponding author. School of Mathematical Sciences,  Key Laboratory of MEA(Ministry of Education) and Shanghai Key Laboratory of PMMP,  East China Normal University, Shanghai 200241, China. Email: \texttt{yyzhang@math.ecnu.edu.cn}. Y. Zhang is sponsored by
  NSFC [No.12271505] and
  STCSM  [No.22DZ2229014].}}
\date{}

\maketitle
\begin{abstract}
We consider a second-order nonlocal parabolic MEMS equation with Dirichlet boundary conditions:
\[
u_t-\Delta u=\frac{\lambda}{(1-u)^2\bigl(1+\int_\Omega\frac{1}{1-u}\,dx\bigr)^2},\quad x\in\Omega,\ t>0,
\]
where \(\Omega\subset\mathbb{R}^N\) \((1\le N\le3)\) is a bounded smooth domain and \(\lambda>0\). Using operator semigroups and the contraction mapping principle, we prove local existence and give a quenching criterion. Under suitable smallness conditions on \(\lambda\) and the initial data, global existence and exponential convergence to the minimal steady state are obtained. Assuming the global solution stays uniformly away from the singularity \(u=1\), we show that the system forms a gradient system. By establishing analyticity of the energy and a Lojasiewicz--Simon inequality, we prove that the solution converges to a steady state with either exponential or algebraic rate depending on the Lojasiewicz exponent. Numerical experiments in 1D and 2D illustrate the results and support conjectures on the \(\lambda\)-dichotomy.

\end{abstract}

\section{Introduction}

\subsection{Research background}

\qquad Micro-Electro-Mechanical Systems (MEMS) are devices that integrate miniature mechanical structures, micro-sensors, micro-actuators, electronic circuits, and other devices. Typically ranging in size from micrometers to millimeters, MEMS offer advantages such as compact dimensions, low energy consumption, high integration, and high sensitivity.

The development history of MEMS can be traced back to the 1950s. In 1959, physicist Richard Feynman proposed the possibility of manipulating and controlling objects at a minuscule scale. In 1964, Westinghouse produced the first batch of mass-produced MEMS devices, known as resonant gate transistors. In 1982, Kurt Petersen published a seminal paper titled "Silicon as a Mechanical Material", which extensively discussed the material properties of silicon and etching data, greatly advancing the field of MEMS.

Today, amid the rapid growth of digitalization, the Internet of Things (IoT), and artificial intelligence (AI) technologies, MEMS is facing unprecedented opportunities for development. MEMS plays an indispensable role in cutting-edge fields such as industrial inspection, autonomous driving, and smart healthcare. Investigating the mathematical models of MEMS will facilitate an in-depth exploration of their properties, provide a robust theoretical foundation for MEMS advancement, and propel MEMS-related technologies toward more advanced levels.

\begin{figure}[H]
    \centering
    \includegraphics[width=10cm,height=6cm]{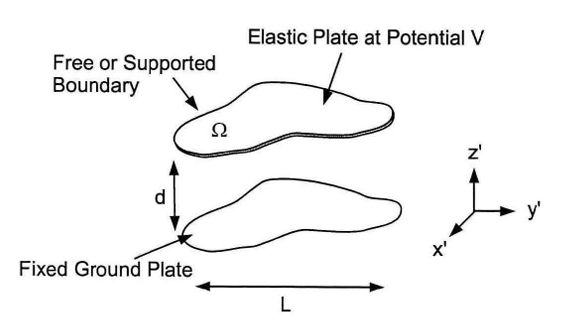}
    \caption{MEMS device(\cite{ref7})}\label{tu1}
\end{figure}

A simple MEMS device is shown in the Figure \ref{tu1}. The upper part consists of an elastic membrane coated with a metallic conductor, while the bottom features a parallel rigid plate. When a voltage is applied, the membrane deflects toward the bottom due to electrostatic forces. When the applied voltage exceeds a certain critical value, the elastic membrane is pulled in and touches the bottom plate, tneh the system loses its stable state. This phenomenon is called quenching.

The mathematical model describing this physical process is given by (\cite{ref7}):
\begin{equation*}
    \begin{cases}
        \rho A u_{tt}+a u_t=T \Delta u-B \Delta^{2} u-\frac{C^{2} V^{2}}{(d+u)^{2}} ,\quad&   x\in \Omega,\ t>0,\\
        -1<u(x, t) \le 0 ,\quad&  x\in \Omega,\ t>0,\\
        u(x, t)=B \frac{\partial u}{\partial \eta}(x, t)=0 ,\quad &  x\in \partial \Omega,\ t>0,\\
        u(x, 0)=u_0 ,\quad u_{t}(x, 0) \ge 0,\quad &  x \in \Omega,
    \end{cases}
\end{equation*}
where $\Omega\subset\mathbb{R}^{N}$ is a bounded domain with smooth boundary, $\rho$ is the membrane density, $A$ is the membrane thickness, $a$ is the damping coefficient, $T$ is the tension constant, $B$ represents the bending energy, $C$ and $V$ denote capacitance and voltage respectively, $d$ is the distance between the membrane and the bottom plate, and $u$ represents the membrane displacement.

\begin{figure}[H]
    \centering
    \includegraphics[width=10cm,height=6cm]{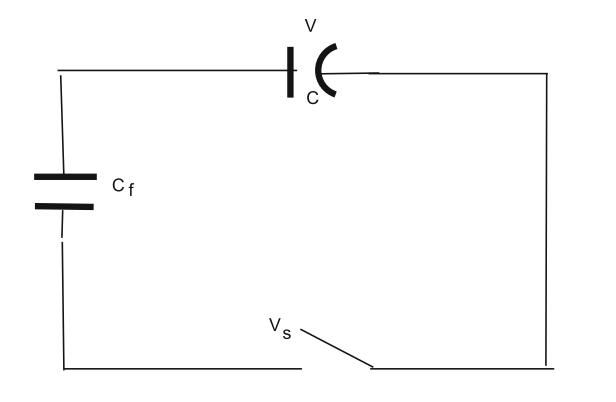}
    \caption{capacitive control circuit (\cite{ref15})}
    \label{tu2}
\end{figure}

In 1997, in order to improve the stability of the device, Seeger and Crary proposed adding a series capacitor to the circuit containing the MEMS device to control the voltage (as shown in Figure \ref{tu2}). According to Kirchhoff's law, we have
$$V=\frac{V_s}{1+\frac{C}{C_f}},$$
where $V$ is the voltage across the MEMS device, $V_s$ is the supply voltage, $C$ is the capacitance of the MEMS device, and $C_f$ is the capacitance of the fixed series capacitor. Note that the capacitance $C$ of the MEMS device increases as the gap between the membrane and the bottom plate decreases, which in turn causes the voltage $V$ to decrease, thereby enhancing system stability. In this case, the equation is modified as follows:
$$
 \rho A u_{tt}+a u_t=T \Delta u-B \Delta^{2} u-\frac{\lambda}{(d+u)^{2}\left(1+\alpha \int_{\Omega} \frac{1}{d+u} d x\right)^{2}},
$$
where $\lambda = C^2 V_s^2$ is the voltage parameter. Note that due to the addition of the series capacitor, the nonlocal term $\frac{1}{(1+\alpha\int_{\Omega}\frac{1}{d+u}\,\mathrm{d}x)^2}$ appears in the equation, which depends on the global value of $u$.

To simplify the problem, we introduce dimensionless variables and adjust parameters to normalize the elastic coefficient $T$, damping coefficient $a$, and gap distance $d$. By neglecting membrane rigidity and thickness (i.e. $A = B = 0$), we obtain the equation:
\begin{equation}
    \noindent
    \begin{cases}
        u_{t}-\Delta u=\frac{\lambda}{(1-u)^2(1+\alpha\int_{\Omega}\frac{1}{1-u}\,\mathrm{d}x)^2},&x\in\Omega,t>0\\
        u=0,&x\in\partial\Omega,t>0\\
        u(x,0)=u_0(x),&x\in\Omega
    \end{cases}\label{1.1}
\end{equation}\\
where $\lambda,\alpha>0$,\ $\Omega\subset\mathbb{R}^{N}$. Without loss of generality, we assume $\alpha = 1$. Equation (\ref{1.1}) is the second order nonlocal parabolic MEMS equation studied in this paper.

If $\alpha = 0$, it becomes the second order local parabolic MEMS equation:
\begin{equation}
    \begin{cases}
        u_{t}-\Delta u=\frac{\mu}{(1-u)^2},&x\in\Omega,t>0\\
        u=0,&x\in\partial\Omega,t>0\\
        u(x,0)=u_0(x),&x\in\Omega
    \end{cases}\label{1.2}
\end{equation}

This paper will study the well-posedness and asymptotic behavior of solutions to equation (\ref{1.1}) through contraction mapping principle, operator semigroup theory, gradient system theory, analyticity theory, and the Lojasiewicz-Simon method, with intuitive demonstrations provided by numerical experiments.

\subsection{Research status}

\qquad At present, for local parabolic MEMS equation (\ref{1.2}) and its corresponding steady-state equation:
\begin{equation}
    \begin{cases}
        -\Delta \phi=\frac{\mu}{(1-\phi)^2},&x\in\Omega\\
        0 \le \phi<1,&x\in\Omega\\
        \phi=0,&x\in\partial\Omega
    \end{cases}\label{1.3}
\end{equation}
the structure and asymptotic behavior of the solution are relatively clear (\cite{ref7}\cite{ref12}\cite{ref15}). For example, the conclusion in \cite{ref7} is that: for $u_0=0$, there is a critical voltage $\mu^{*}>0$ such that when $\mu<\mu^{*}$, equation (\ref{1.2}) has a unique solution $u(x,t)$ and $u$ converges monotonically increasing to the minimal classical solution of the steady-state equation (\ref{1.3}) as $t\to +\infty$; when $\mu>\mu^{*}$, the solution of equation (\ref{1.2}) quenches in finite time, and the steady-state equation (\ref{1.3}) has no solution.

For nonlocal parabolic MEMS equation (\ref{1.1}) and its corresponding steady-state equation:
\begin{equation}
    \begin{cases}
        -\Delta \phi=\frac{\lambda}{(1-\phi)^2(1+\int_{\Omega}\frac{1}{1-\phi}\,\mathrm{d}x)^2},&x\in\Omega\\
        0 \le \phi<1,&x\in\Omega\\
        \phi=0,&x\in\partial\Omega
    \end{cases}\label{1.4}
\end{equation}
in the one-dimensional case ($N=1$, $\Omega=[-1,1]$), reference \cite{ref1} concludes by ODE methods that: there exists $\lambda^{*} >0$ such that when $\lambda<\lambda^{*}$, the steady-state equation (\ref{1.4}) has two solutions; when $\lambda=\lambda^{*}$, (\ref{1.4}) has a unique solution; and when $\lambda>\lambda^{*}$, (\ref{1.4}) has no solution. It also states that: when $0<\lambda<9\mu^{*}$ ($\mu^{*}$ is the critical voltage for the local equation) and $u_0=0$, equation (\ref{1.1}) has a global solution $u$, which converges to the minimal classical solution of the steady-state equation (\ref{1.4}) as $t\to +\infty$.

Reference \cite{ref2} continues the research of \cite{ref1} in higher dimensions ($N\ge 2$). Using the lower-upper solution method, it proves the local existence and uniqueness of solutions to equation (\ref{1.1}) and obtains the following results: when $0<\lambda<(1+|\Omega|)^2\mu^{*}$ ($\mu^{*}$ is the critical voltage for the local equation), the steady-state equation (\ref{1.4}) has classical solutions, and let $\omega_{\lambda}$ denote the minimal classical solution among them, if $u_0\le \omega_{\lambda}$, equation (\ref{1.1}) has a global solution $u$ that converges to $\omega_{\lambda}$ as $t\to +\infty$. Additionally, reference \cite{ref2} obtains some results concerning quenching: assuming $|\Omega|\le \frac{1}{2}$, then for any given $\lambda>0$, there exists initial data $u_0$ such that the solution of equation (\ref{1.1}) quenches in finite time. Reference \cite{ref12} proves that when $|\Omega|<1$ and $u_0=0$, there exists $\lambda$ sufficiently large such that the solution of equation (\ref{1.1}) quenches in finite time.

Particularly, when $\Omega=B_1(0)\subset\mathbb{R}^{N}$, the solutions of the steady-state equation (\ref{1.4}) are radially symmetric. Reference \cite{ref2} studies the solution structure of the steady-state equation (\ref{1.4}) in this case. Reference \cite{ref5} proves: for radially symmetric and monotonically decreasing initial data $u_0(r)$, when $\lambda>\lambda^{*}$, the solution of equation (\ref{1.1}) quenches in finite time and touches 1 only at the point $r=0$, where $\lambda^{*}=\sup\{\lambda>0\mid \text{equation (\ref{1.4}) has a classical solution}\}$.

In fact, for the second order local MEMS equation, the comparison principle holds, leading to favorable properties. However, for the second order nonlocal MEMS equation, the comparison principle fails, which creates difficulties in the analysis. Some conclusions in the aforementioned literature are established by connecting to the local equation by $\mu=\frac{\lambda}{(1+\int_{\Omega}\frac{1}{1-u}\,\mathrm{d}x)^2}$, but this only allows for a rough estimate of the range of $\lambda$. Therefore, this paper attempts to study the nonlocal MEMS equation using the Lojasiewicz-Simon method, which does not rely on the comparison principle.

References \cite{ref13}, \cite{ref9}, \cite{ref16}, \cite{ref8} all introduce the main content of the Lojasiewicz-Simon method and apply it to different equations. For example, reference \cite{ref16} studies the nonlocal biharmonic equation with Navier boundary conditions
$$u_{tt}+u_{t}+\Delta^2 u=\left(\beta \int_{\Omega}|\nabla u|^2\,\mathrm{d}x+\gamma\right)\Delta u+\frac{\lambda}{(1-u)^\sigma\left(1+\alpha\int_{\Omega}\dfrac{1}{(1-u)^{\sigma-1}}\,\mathrm{d}x\right)^\sigma},$$
proving that when the initial data and initial energy satisfy certain conditions, the equation has a global solution $u$ in $H^2\cap H_0^1$, and uses the Lojasiewicz-Simon method to prove that the global solution $u$ converges to a steady-state solution in $H^2\cap H_0^1$.

Currently, the well-posedness and convergence problem for solutions of the second order nonlocal parabolic MEMS equation (\ref{1.1}) remains unresolved. This paper will first construct a suitable function space, and with the help of operator semigroup theory and contraction mapping principle, we will obtain the local existence of the solution to equation (\ref{1.1}) in $H^2\cap H_0^1$, and obtain the global existence of the solution to equation (\ref{1.1}) under certain conditions. Then, based on the establishment of local existence, this paper will use the Lojasiewicz-Simon method to prove that: if a global solution exists and this global solution maintains a uniform distance from 1 with respect to time $t$, then the global solution converges to a steady-state solution in $H^2\cap H_0^1$.

\subsection{Main conclusion}

\qquad This paper considers the second order nonlocal parabolic MEMS equation (\ref{1.1}) with Dirichlet boundary conditions. First, we establish local existence of solutions using operator semigroup theory and contraction mapping principle.
\begin{theorem}\label{local existence}
Let $\Omega\subset\mathbb{R}^{N}$ and $u_0 \in H^{2}\cap H_{0}^{1}(\Omega)$, $\|u_{0}\|_{L^{\infty}(\Omega)} \le 1-2 \delta$, where $\delta \in\left(0, \frac{1}{2}\right)$, and $u_0$ is continuous on $\bar{\Omega}$. Then for any $\lambda>0$, there exists $\widetilde{T}>0$ depending on $\Omega$, $\delta$, and $\lambda$ such that when $0<T< \widetilde{T}$, equation (\ref{1.1}) has a unique solution $u(t, x)$ satisfying
$$
u \in C\left([0, T]; H^{2} \cap H_{0}^{1}(\Omega)\right) \cap C^{1}\left([0, T]; L^{2}(\Omega)\right).
$$
Moreover, there exists $T^{*}>0$ such that $[0,T^{*})$ is the maximal existence interval, and either $T^{*}=\infty$, or $T^{*}<\infty$ with
$$
\sup_{(t, x) \in [0, T^{*}) \times \Omega} u(t, x)=1.
$$
\label{local soultion}
\end{theorem}

We will utilize the quenching results obtained in Theorem \ref{local existence}, through a series of estimates, further establish the global existence of solutions under certain conditions.

\begin{theorem}\label{global existence}
Let $\Omega\subset\mathbb{R}^{N}$. Assume $u_{0}\in H^{2}\cap H_{0}^{1}(\Omega)$ is continuous on $\bar{\Omega}$ and satisfies: for some $\delta \in\left(0, \frac{1}{2}\right)$,
\begin{equation}
0 \le u_{0}(x) \le 1-2 \delta ,\quad 0 \le w_{\lambda}(x) \le 1-2 \delta,\label{029}
\end{equation}
where $w_{\lambda}$ is the minimal steady-state solution of (\ref{1.4}). Then there exist positive constants $\theta$ and $\lambda_{0}$ depending on $\Omega$, $N$, $\delta$ such that if additionally
\begin{equation}
\left\|u_{0}-w_{\lambda}\right\|_{H^{2}(\Omega)} \le \theta,\quad \lambda \le \lambda_{0} , \label{030}
\end{equation}
equation (\ref{1.1}) has a unique global solution satisfying:
\begin{equation}
u \in C\left([0, \infty), H^{2}(\Omega) \cap H_{0}^{1}(\Omega)\right) \cap C^{1}\left([0, \infty), L^{2}(\Omega)\right),\label{031}
\end{equation}
and there exists $\alpha > 0$ depending on $\Omega$, $N$, $\delta$ such that:
\begin{equation}
\left\|u(t)-w_{\lambda}\right\|_{H^{2}(\Omega)} \le C e^{-\alpha t}.\label{032}
\end{equation}
\end{theorem}

Based on the local existence of solutions, this paper assumes that equation (\ref{1.1}) has a unique global solution $u$ satisfying
\begin{equation}
    \begin{aligned}
        u\in C([0,+\infty),H^2\cap H_0^1(\Omega))\cap C^1([0,+\infty),L^2(\Omega)),\quad\ \ \\
        \|u\|_{L^\infty}\le 1-\delta,\ \ \forall t\ge 0,\ \ \ \  \text{where}\ \delta\in(0,1)\ \text{depends on}\ \lambda,\ u_0,\ \Omega.
    \end{aligned}\label{u}
\end{equation}
and studies the asymptotic behavior of this global solution.

\begin{theorem}\label{theorem1.1}
  Let $\Omega\subset\mathbb{R}^{N}\ (1\le N\le 3)$ be a bounded domain with smooth boundary. For any $\lambda>0$ and initial data $u_0\in H^2\cap H_0^1(\Omega)$, if equation (\ref{1.1}) has a global solution $u\in C([0,+\infty),H^2\cap H_0^1(\Omega))\cap C^1([0,+\infty),L^2(\Omega))$ satisfying $\|u\|_{L^\infty}\le 1-\delta$, $\forall \ t\ge 0$, where $\delta\in (0,1)$, then there exists $\psi\in C^\infty$ satisfying
  \begin{equation}
      \begin{cases}
          -\Delta\psi=\frac{\lambda}{(1-\psi)^2(1+\int_{\Omega}\frac{1}{1-\psi}\,\mathrm{d}x)^2},\ x\in\Omega \\
          \psi |_{\partial\Omega}=0 \label{1.5}
      \end{cases}
  \end{equation}
  such that
  \begin{equation}
      \lim_{t \to +\infty} \|u(\cdot,t)-\psi\|_{H^2}=0. \label{1.6}
  \end{equation}
\end{theorem}

Theorem \ref{tuilun} in Section 2.6 shows that when $\lambda$ is sufficiently large, the steady-state equation (\ref{1.5}) has no solution. Combining this with Theorem \ref{theorem1.1}, we obtain the following corollary:

\begin{corollary}\label{tuilun1}
Let $\Omega\subset\mathbb{R}^{N}\ (N=2,3)$ be a strictly star-shaped domain, and let $\beta$ be defined as in Theorem \ref{tuilun}. If $$\lambda>\frac{N(1+|\Omega|)^4}{2\beta |\Omega|^2},$$
then for any initial data $u_0\in H^2\cap H_0^1(\Omega)$, equation (\ref{1.1}) has no global solution satisfying (\ref{u}).
\end{corollary}

\begin{remark}
    In particular, when $\Omega=B_1(0)\subset\mathbb{R}^{N}\ (N=2,3)$, if $$\lambda>\frac{N^2(1+|\Omega|)^4}{2|\Omega|},$$
where $|\Omega|$ is the volume of the unit ball $B_1(0)$, then for any initial data $u_0\in H^2\cap H_0^1(\Omega)$, equation (\ref{1.1}) has no global solution satisfying (\ref{u}).
\end{remark}

\begin{theorem}\label{theorem1.2}
    Let $\theta\in (0,\frac{1}{2}]$ be defined as in Theorem \ref{theorem2.1}. If $\theta=\frac{1}{2}$, then the solution $u$ of equation (\ref{1.1}) decays exponentially to $\psi$ in $H^2\cap H_0^1(\Omega)$, i.e.
    \begin{equation}
        \|u(\cdot,t)-\psi\|_{H^2}\le C_0e^{-C_1t},\quad t\to +\infty,\label{2.33}
    \end{equation}
    where $C_0,\ C_1>0$ are constants.
    If $\theta\in (0,\frac{1}{2})$, then $u$ decays polynomially to $\psi$ in $H^2\cap H_0^1(\Omega)$, i.e.
    \begin{equation}
        \|u(\cdot,t)-\psi\|_{H^2}\le C(1+t)^{-\frac{\theta}{1-2\theta}},\quad t\to +\infty,\label{2.34}
    \end{equation}
    where $C>0$ is a constant.
\end{theorem}

The structure of this paper is organized as follows: Section 2 introduces contraction mapping principle, operator semigroups, analyticity, semi-Fredholm operators, Lojasiewicz Inequality and so on. Section 3 proves the local existence of solutions to equation (\ref{1.1}). Section 4 proves the existence of global solutions to equation (\ref{1.1}) under certain conditions. Section 5.1 constructs the energy functional and gradient system for equation (\ref{1.1}); Section 5.2 proves the analyticity of the energy functional; Section 5.3 extends the Lojasiewicz Inequality to the problem studied in this paper; Section 5.4 establishes the convergence of solutions to equation (\ref{1.1}); Section 5.5 derives the convergence rate of solutions to equation (\ref{1.1}). Section 6 provides numerical experiments to visually demonstrate the asymptotic behavior and quenching phenomena of solutions, and proposes several conjectures.

\section{Preliminaries}

\subsection{Contraction mapping principle}

\begin{theorem}[\cite{rud}]
    Let $(X;\|\cdot\|)$ be a Banach space, $T:X\to X$, and $\exists L\in [0, 1)$ such that $\forall x$, $y\in X$, $\|Tx-Ty\|\le L\|x-y\|$. Then $T$ has a unique fixed point $x^{*}\in X$, such that $Tx^{*}=x^{*}$.
\end{theorem}

\subsection{operator semigroup}

\begin{definition}[See Definition 2.1 in \cite{ref18}]
Let $H$ be a Banach space, consider a one-parameter family of bounded linear operators $\{S(t) ; t \geq 0\}$ on $H$. If it satisfies:
\\(1) $S(0)=I$ (the identity operator);
\\(2) $\forall\ t_1,t_2>0$, $S(t_1+t_2)=S(t_1)S(t_2)=S(t_2)S(t_1)$;
\\(3) For any $u_{0} \in H$, $\lim\limits_{t\to 0}S(t) u_{0} =u_0$,
\\then it is called a $C_0$-semigroup on $H$. Furthermore, define:
$$
\begin{aligned}
D(A):&=\left\{x \in H \mid \lim _{h \rightarrow+0} \frac{S(h) x-x}{h}\ \text{exists} \right\}, \\
A x:&=\lim _{h \rightarrow+0} \frac{S(h) x-x}{h} ,\quad \forall x \in D(A),
\end{aligned}
$$
The operator $A$ is called the infinitesimal generator of this $C_0$-semigroup $\{S(t) ; t \geq 0\}$.
\end{definition}

\begin{definition}[See p. 238 in \cite{ref8}]
    Let $H$ be a Banach space. For an initial value $u_0\in H$, the set $\bigcup_{t\ge 0}S(t)u_0=\bigcup_{t\ge 0}u(t)$ is called the orbit starting from $u_0$, and its $\omega$-limit set is defined as
    $$\omega(u_0)=\bigcap_{s\ge 0}\overline{\bigcup_{t\ge s}S(t)u_0}=\{\psi\ |\ \exists t_n,\ \text{such that}\  u(\cdot,t_n)\to \psi,\ t_n\to +\infty\}.$$
    If $x\in H$ satisfies $S(t)x=x$ for all $t\ge 0$, then $x$ is called a fixed point of $S(t)$; if $X\subset H$ satisfies $S(t)X=X$ for all $t\ge 0$, then $X$ is called an invariant set.
\end{definition}

\begin{definition}[\cite{are}]
    Consider the equation:
    \begin{equation*}
    \left\{\begin{array}{l}
    u_t+A u=F(u), \\
    u(0)=u_{0}.
    \end{array}\right.
    \label{mild}
    \end{equation*}
    where $A$ is defined on a dense subset $D(A)$ of Banach space $B$, generating a $C_0$-semigroup $\{S(t) ; t \geq 0\}$. If  $u_0\in B$, and $u$ satisfies
    $$ u(t)=S(t) u_{0}+\int_{0}^{t} S(t-\tau) F(\tau) d \tau ,$$
    then $u$ is called a mild solution of equation.
    \end{definition}

\begin{definition}[\cite{ref18}]
    For equation
    \begin{equation}
        \left\{\begin{array}{l}u_t+A u(t)=f(t), \quad t>t_{0} , \\ u\left(0\right)=u_{0},
    \end{array}\right.\label{0021}
    \end{equation}
    where $-A$ is the infinitesimal generator of the $C_0$-semigroup $\{ T(t),\ t \geq 0\}$ on a reflexive Banach space $X$, $f:\left[t_{0}, T\right]  \rightarrow X $. If $u$ is differentiable almost everywhere on $[0,T]$, $u_t\in L^1(0,T;X)$, and satisfies $u\left(0\right)=u_{0}$ and $u_t+A u(t)=f(t)$ to be true almost everywhere on $[0,T]$, then $u$ is called a strong solution of the equation.
\end{definition}

\begin{lemma}[\cite{ref18}]\label{classic exists}
    If $f:\left[t_{0}, T\right]  \rightarrow X $ satisfies the Lipschitz condition, then $ \forall u_0\in D(A) $, the equation (\ref{0021}) has a unique classical solution:
   \begin{equation*}
       u(t)=T(t) u_0+\int_{0}^{t} T(t-s) f(s) ds.
   \end{equation*}
\end{lemma}

\begin{lemma}[\cite{ref18}]\label{strong exists}
    For equation
    \begin{equation}
	\left\{\begin{array}{l}u_t+A u(t)=f(t, u(t)), \quad t>t_{0},  \\ u\left(t_{0}\right)=u_{0},\end{array}\right.\label{0022}
    \end{equation}
    If $f:\left[t_{0}, T\right] \times X \rightarrow X $ satisfies the Lipschitz condition with respect to $t$ and $u$, $u_{0} \in D(A)$ and $u$ is a mild solution to the equation (\ref{0022}), then $u$ is a strong solution. At this time, $u$ satisfies the Lipschitz condition with respect to $t$.
\end{lemma}

\begin{lemma}[\cite{ref18}]\label{classic exists 2}
   If $f:\left[t_{0}, T\right] \times X \rightarrow X $ is a continuously differentiable function, $u_{0} \in D(A)$ and $u$ is a mild solution to the equation (\ref{0022}), then $u$ is a classical solution.
\end{lemma}

\subsection{Analytic function and analytic mapping}

\begin{definition}[See Definition 2.2.3 in \cite{ref21}]
    Let $f$ be a smooth function defined on $D\subset \mathbb{R}^d$ and $y_0\in D$. If there exists $r\in (0,\mathrm{dist}(y_0,\partial D))$ such that
    $$f(y)=\sum\limits_{\alpha\in\mathbb{Z}_{+}^d}\frac{1}{\alpha !}\partial^{\alpha}f(y_0)(y-y_0)^{\alpha},\quad \forall y\in B_r(y_0),$$
    where
    $$\mathbb{Z}_{+}^d=\{\alpha=(\alpha_1,\dots,\alpha_d)\mid \alpha_i\ge 0,\ i=1,\dots,d\},$$
    then $f$ is called analytic at $y_0$. Furthermore, if $f$ is analytic at every point of $D$, it is called analytic on $D$.
\end{definition}

\begin{definition}[See Definition 2.2.6 in \cite{ref21}]
    Let $\Omega\subset\mathbb{R}^n$ be a bounded domain, $D\subset \mathbb{R}^d$, and $f=f(x,y):\Omega\times D\to\mathbb{R}$ a smooth function such that for every $D'\subset\subset D$, $f$ is bounded on $\Omega\times D'$. If for every $D'\subset\subset D$, there exists $r\in (0,\mathrm{dist}(D',\partial D))$ such that
    $$f(x,y+v)=\sum\limits_{\alpha\in\mathbb{Z}_{+}^d}\frac{1}{\alpha !}\partial_y^{\alpha}f(x,y)v^{\alpha},\quad \forall x\in\Omega,\ y\in D',\ v\in B_r(0)\subset \mathbb{R}^d,$$
    then $f$ is called uniformly analytic in $y\in D$ with respect to $x$.
\end{definition}

\begin{definition}[See Definition 2.3.1 in \cite{ref13}]
    Let $X$ and $Y$ be Banach spaces, $a\in X$, and $U\subset X$ an open neighborhood of $a$. For a mapping $f:U\to Y$, if there exist $r>0$ and a sequence of continuous symmetric $n$-linear mappings $(M_n)_{n\in \mathbb{N}}$ such that:\\
    (1) $\sum\limits_{n=1}^{\infty}\|M_n\|_{\mathcal{L}_n(X,Y)}r^n<\infty$, where $\|M_n\|_{\mathcal{L}_n(X,Y)}=\sup\{\|M_n(x_1,\dots,x_n)\|_Y\mid \|x_i\|_X\leq1,\ 1\leq i\leq n\}$;\\
    (2) $\overline{B}(a,r)\subset U$;\\
    (3) For every $h\in\overline{B}(0,r)$, $f(a+h)=f(a)+\sum\limits_{n=1}^{\infty}M_n(h,\dots,h)$.\\
    then $f$ is called analytic at $a$. Furthermore, if $f$ is analytic at every point of $U$, it is called analytic on $U$.
\end{definition}

\begin{theorem}[See Proposition 2.3.4 in \cite{ref13}]
    Let $f\in C^1(U,Y)$. Then
    $$f:U\to Y\ \text{is analytic on}\ U\ \Longleftrightarrow\ Df:U\to \mathcal{L}(U,Y)\ \text{is analytic on}\ U.$$
\end{theorem}

\begin{theorem}[See Theorem 2.3.5 in \cite{ref13}]\label{lemma2.2}
    Let $Z$ be a Banach space, $f$ analytic at $a$, $V$ an open neighborhood of $f(a)$, and $g:V\to Z$ analytic at $f(a)$. Then the composition $g\circ f$ is analytic at $a$.
\end{theorem}

\begin{theorem}[See Lemma 2.2.14 in \cite{ref21}]\label{lemma2.3}
    Let $\Omega\subset \mathbb{R}^n$ be a bounded domain, $D\subset \mathbb{R}^d$, and $f=f(x,u):\Omega\times D\to\mathbb{R}$ a smooth function that is uniformly analytic in $u\in D$ with respect to $x$. Define $U=\{u\in L^{\infty}(\Omega;\mathbb{R}^d)\mid u(\Omega)\subset\subset D\ \mathrm{a.e.}\}$ and the mapping $S:U\to \mathbb{R}$ by
    $$S(u)=\int_{\Omega}f(x,u(x))\,\mathrm{d}x.$$
    Then $S:U\to \mathbb{R}$ is analytic.
\end{theorem}

\subsection{Semi-Fredholm Operators}

\begin{definition}[See Definition 2.2.1 in \cite{ref13}]
    Let $W$ and $Z$ be Banach spaces, and let $A \in \mathcal{L}(W, Z)$ be a bounded linear operator. If the following conditions hold:
    \\(1) $\ker(A)$ is finite-dimensional; \\
    (2) $\mathrm{R}(A)$ is closed, \\
    then $A$ is called a semi-Fredholm operator.
\end{definition}

\begin{remark}[See Remark 2.2.2 in \cite{ref13}]\label{ker}
    The finite dimensionality of $\ker(A)$ implies that there exists a closed subspace $X \subset W$ such that $W = \ker(A) \oplus X$. Furthermore, $R(A) = A(X)$ forms a Banach space under $\|\cdot\|_{Z}$.
\end{remark}

\begin{theorem}[See Theorem 2.2.3 in \cite{ref13}]
    Let $W$ and $Z$ be Banach spaces, $A \in \mathcal{L}(W, Z)$, and $X$ as defined in Remark \ref{ker}. If $\ker A$ is finite-dimensional, then $A: W \to Z$ is a semi-Fredholm operator if and only if there exists a constant $C > 0$ such that for all $u \in X$,
    $$\|Au\|_{Z} \geq C\|u\|_{W}.$$
\end{theorem}

\begin{remark}\label{delta}
    Since $\|u\|_{H^2} \le C\|-\Delta u\|_{L^2}$ (see formula (1.3.26) in \cite{ref8}), the operator $-\Delta: H^2 \cap H_0^1(\Omega) \to L^2(\Omega)$ is semi-Fredholm.
\end{remark}

\begin{theorem}[See Theorem 2.2.5 in \cite{ref13}]\label{lemma2.4}
    Let $A: W \to Z$ be a semi-Fredholm operator and $G: W \to Z$ a compact operator. Then $A + G: W \to Z$ is also a semi-Fredholm operator.
\end{theorem}

\subsection{Lojasiewicz Inequality}

\begin{theorem}[\cite{ref19}]\label{L-S}
    Let $f:\mathbb{R}^{m}\longrightarrow \mathbb{R}$ be analytic in a neighborhood of a critical point $a$ (i.e., $\nabla f(a)=0$). Then there exist constants $\theta\in (0,\frac{1}{2}]$, $\sigma>0$, and $C>0$, depending only on $f$ and $a$, such that whenever $\|x-a\|_{\mathbb{R}^{m}}\le \sigma$,
    $$|f(x)-f(a)|^{1-\theta}\le C\|\nabla f(x)\|_{\mathbb{R}^{m}}.$$
\end{theorem}

This paper extends Theorem \ref{L-S} to obtain a Lojasiewicz-Simon Inequality for the energy functional of equation (\ref{1.1}), with detailed proofs provided in Section 5.3.

\subsection{Steady State equation}

\qquad The existence of solutions to the steady state equation (\ref{1.4}) has been discussed in Section 1.2 (Research status), where classical solutions exist when $\lambda$ is sufficiently small. This section presents a theorem showing that no solution exists when $\lambda$ is large.

\begin{theorem}[See Theorem 3.2 in \cite{ref2}]\label{tuilun}
    Let $\Omega\subset\mathbb{R}^{N}\ (N\ge 2)$ be a strictly star-shaped domain, i.e., there exists $\beta>0$ such that
    $$\nu(x)\cdot x\ge \beta\int_{\partial\Omega}\mathrm{d}x,\quad \forall x\in \partial\Omega,$$
    where $\nu(x)$ is the unit outward normal vector at $x$. If $$\lambda>\frac{N(1+|\Omega|)^4}{2\beta |\Omega|^2},$$
    then the steady state equation (\ref{1.4}) has no solution.
\end{theorem}

\section{Local existence}

\qquad In this section, we prove the Theorem \ref{local existence}. We will truncate the equation to obtain the auxiliary equation, then construct the function space and use the contraction mapping principle to study the local existence of the solution of the auxiliary equation, and finally get the conclusion from the sobolev embedding theorem.

Denote $ A=-\triangle $, then $D(A)=H^{2}\cap H_{0}^{1}(\Omega)$, and $A$ generates a $C_0$-semigroup on $L^{2}(\Omega)$ (see the proof of Theorem 2.7.2 in \cite{ref8}), which is denoted as $\{Q(t);t \geq 0\}$. Thus we consider:
\begin{equation*}
u(t)=Q(t) u_{0}+\int_{0}^{t} Q(t-s) f(u(s)) \mathrm{d} s,
\end{equation*}

Note that $f(u)$ has a singular point $u=1$, so we consider the truncated function
\begin{equation*}
 g _{\delta}(u):=\left\{\begin{array}{ll}\frac{1}{1-u}, & \quad u \le 1-\delta ,
\\ \delta^{-1}, & \quad u>1-\delta,
\end{array}\right.
\end{equation*}
and
\begin{equation}
f_{\delta}(u):=\lambda\left(g_{\delta}(u)\right)^{2}\left(1+\int_{\Omega} g_{\delta}(u) d x\right)^{-2}.\label{048}
\end{equation}
Thus we have $f_{\delta} \in W^{1, \infty}(\mathbb{R})$ and
\begin{equation}
\left\|g_{\delta}\right\|_{L^{\infty}(\mathbb{R})} \le \delta^{-1}, \left\|g_{\delta}^{\prime}\right\|_{L^{\infty}(\mathbb{R})} \le \delta^{-2}.
\label{013}
\end{equation}

Now we consider the following auxiliary equation
\begin{equation}
\left\{\begin{array}{lll}
 u_{t}-\triangle u=f_{\delta}(u), \quad & (t,x)\in (0, \infty)\times\Omega  ,
\\  u(t, x)=0, \quad   &(t,x)\in (0, \infty)\times \partial \Omega,
\\ u(0, x)=u_{0}(x), \quad  &x\in  \Omega ,
\end{array}\quad \right. \label{014}
\end{equation}
and operator $B$
\begin{equation}
\begin{array}{lll}
&L^2(\Omega) &\to L^2(\Omega),
\\ &u  &\mapsto  Q(t) u_{0}+\int_{0}^{t} Q(t-s) f_{\delta}(u(s)) \mathrm{d} s.
\end{array}
\end{equation}

\begin{lemma}
    Let $$ Y_{M}:=\left\{u \in C\left([0, T] ; L^{2}(\Omega)\right), \|u\|_{Y}\le  M \right\} ,\quad  $$
    where $\|u\|_{Y}=\|u\|_{L^{\infty}\left([0, T] ; L^{2}(\Omega)\right)}$, $M :=\left\|u_{0}\right\|_{L^{2}} +1$. Then $f_\delta$ (as defined in (\ref{048})) is Lipschitz continuous in $Y_M$.
\end{lemma}

\begin{proof}
    Note that
    $\forall u_1, u_2\in Y_M$, we have
    \begin{equation}
    \begin{aligned}
    &\left|f_{\delta}(u_1)-f_{\delta}\left(u_2\right)\right|
    \\=&\left|\frac{\lambda g_{\delta}^{2}(u_1) }{\left(1+\int_{\Omega} g_{\delta}(u_1) \mathrm{d} x\right)^{2}}-\frac{\lambda g_{\delta}^{2}(u_2) }{\left(1+\int_{\Omega} g_{\delta}(u_2) \mathrm{d} x\right)^{2}}\right|
    \\=&\left|\frac{\lambda\left( g_{\delta}^{2}(u_1) -
     g_{\delta}^{2}(u_2) + g_{\delta}^{2}(u_2) \right)}{\left(1+\int_{\Omega} g_{\delta}(u_1) \mathrm{d} x\right)}-\frac{\lambda g_{\delta}^{2}(u_2) }{\left(1+\int_{\Omega} g_{\delta}(u_2) \mathrm{d} x\right)^{2}}\right|
    \\=  &\left| \frac{\lambda\left(g_{\delta}(u_1)+g_{\delta}\left(u_2\right)\right)}{\left(1+\int_{\Omega} g_{\delta}(u_1) \mathrm{d} x\right)^{2}} \cdot\left(g_{\delta}(u_1)-g_{\delta}\left(u_2\right)\right) \right.
    \\ -&\frac{\lambda g_{\delta}^{2}\left(u_2\right) \cdot \int_{\Omega}\left(g_{\delta}(u_1)-g_{\delta}\left(u_2\right)\right) \mathrm{d} x}{\left(1+\int_{\Omega} g_{\delta}(u_1) \mathrm{d} x\right)^{}\left(1+\int_{\Omega} g_{\delta}\left(u_2\right) \mathrm{d} x\right)^{2}}
    \\-&\left. \frac{\lambda g_{\delta}^{2} \left(u_2\right)  \cdot \int_{\Omega}\left(g_{\delta}(u_1)-g_{\delta}\left(u_2\right)\right) \mathrm{d} x}{\left(1+\int_{\Omega} g_{\delta}(u_1) \mathrm{d} x\right)^{2}\left(1+\int_{\Omega} g_{\delta}\left(u_2\right) \mathrm{d} x\right)^{}} \right|
    \\ \le &2 \lambda\left(\delta^{-3}\left|u_1-u_2\right|+\delta^{-4} \int_{\Omega}\left|u_1-u_2\right| \mathrm{d} x\right),
    \end{aligned}\label{015}
    \end{equation}
    then
    \begin{equation*}
    \begin{aligned}
    &\left\|f_{\delta}(u_1)-f_{\delta}\left(u_2\right)\right\|_{L^{2}(\Omega)}
    \\ = &  \left\| f_{\delta}(u_1(s)) - f_{\delta}(u_2(s))  \right\|_{ L^{2}\left(\Omega\right)}
    \\ \le &   \left\|2 \lambda\left(\delta^{-3}\left|u_1-u_2\right|+\delta^{-4} \int_{\Omega}\left|u_1-u_2\right| \mathrm{d} x\right)\right\|_{ L^{2}\left(\Omega\right)}
    \\ \le &    \left\|2 \lambda\delta^{-3}\left|u_1-u_2\right|\right\|_{ L^{2}\left(\Omega\right)} + \left\|2 \lambda\delta^{-4} \int_{\Omega}\left|u_1-u_2\right| \mathrm{d} x\right\|_{ L^{2}\left(\Omega\right)}
    \\ \le &   2\lambda\delta^{-3}\left\|u_1-u_2\right\|_{ L^{2}\left(\Omega\right)} + 2\lambda\delta^{-4}|\Omega|\left(\int_{\Omega}\left|u_1-u_2\right|^2 \mathrm{d} x\right)^{\frac{1}{2}}
    \\ = &    \left( 2\lambda\delta^{-3} + 2\lambda\delta^{-4}|\Omega|\right)\left\|u_1-u_2\right\|_{ L^{2}\left(\Omega\right)}.
    \end{aligned}
    \end{equation*}
    The lemma is proved.
\end{proof}

On this basis, the classical theory of evolution equations (such as Theorem 2.5.4 of \cite{ref8}) can guarantee the existence of solutions on a certain interval. In order to more clearly show the existence interval of local solutions, we give a detailed compression process in the proof of the following proposition.

\begin{proposition}
    $\exists\  T^*>0$, when $T<T^*$, equation (\ref{1.4}) has a unique classical solution on $[0,T]$.
    \label{lem 01}
\end{proposition}

\begin{proof}
    $\forall u\in Y_M$, we have
    \begin{equation*}
    \begin{aligned}
    &\left\|Bu\right\| _{L^{\infty}\left([0, T] ; L^{2}(\Omega)\right)}=\left\|Q(t) u_{0}+\int_{0}^{t} Q(t-s) f_{\delta}(u(s)) \mathrm{d} s\right\| _{L^{\infty}\left([0, T] ; L^{2}(\Omega)\right)}
    \\ \le &  \left\|u_{0}\right\| _{ L^{2}\left(\Omega\right)}+\left\|\int_{0}^{t} Q(t-s) f_{\delta}(u(s)) \mathrm{d} s\right\|_{L^{\infty}\left([0, T] ; L^{2}(\Omega)\right)}
    \\ \le & \left\|u_{0}\right\|_{ L^{2}\left(\Omega\right)}+\left\|\int_{0}^{t} \left\|Q(t-s) f_{\delta}(u(s))\right\|_{ L^{2}\left(\Omega\right)} \mathrm{d} s \right\|_{L^{\infty}\left([0, T] \right)}
    \\ \le & \left\|u_{0}\right\|_{ L^{2}\left(\Omega\right)}+\left\|\int_{0}^{t} \left\|f_{\delta}(u(s))\right\|_{ L^{2}\left(\Omega\right)} \mathrm{d} s \right\|_{L^{\infty}\left([0, T] \right)}
    \\= & \left\|u_{0}\right\|_{ L^{2}\left(\Omega\right)}+\left\|\int_{0}^{t} \left\|\lambda\left(g_{\delta}(u)\right)^{2}\left(1+\int_{\Omega} g_{\delta}(u) d x\right)^{-2}\right\|_{ L^{2}\left(\Omega\right)} \mathrm{d} s \right\|_{L^{\infty}\left([0, T] \right)}
    \\ \le & \left\|u_{0}\right\|_{ L^{2}\left(\Omega\right)}+\left\|\int_{0}^{t} \left\|\lambda\left(\delta^{-1}\right)^{2}\left(1+\delta^{-1}|\Omega|\right)^{-2}\right\|_{ L^{2}\left(\Omega\right)} \mathrm{d} s \right\|_{L^{\infty}\left([0, T] \right)}
    \\ \le & \left\|u_{0}\right\|_{ L^{2}\left(\Omega\right)}
    +T|\Omega|\lambda\delta^{-2}
    \left(1+\delta^{-1}|\Omega|\right)^{-2},
    \end{aligned}
    \end{equation*}
    on the other hand, $\forall u_1, u_2\in Y_m$, we have
    \begin{equation*}
    \begin{aligned}
    &\left\|Bu_1-Bu_2\right\|_{L^{\infty}\left([0, T] ; L^{2}(\Omega)\right)}
    \\=&\left\|\left(Q(t) u_{0}+\int_{0}^{t} Q(t-s) f_{\delta}(u_1(s)) \mathrm{d} s\right)-\left(Q(t) u_{0}+\int_{0}^{t} Q(t-s) f_{\delta}(u_2(s)) \mathrm{d} s\right)\right\|_{L^{\infty}\left([0, T] ; L^{2}(\Omega)\right)}
    \\=&\left\|\int_{0}^{t} Q(t-s) f_{\delta}(u_1(s)) -Q(t-s) f_{\delta}(u_2(s)) \mathrm{d} s\right\|_{L^{\infty}\left([0, T] ; L^{2}(\Omega)\right)}
    \\ \le &  \left\|\int_{0}^{t} \left\|Q(t-s) \left( f_{\delta}(u_1(s)) - f_{\delta}(u_2(s))\right)  \right\|_{ L^{2}\left(\Omega\right)} \mathrm{d} s \right\|_{L^{\infty}\left([0, T] \right)}
    \\ \le &  \left\|\int_{0}^{t} \left\| f_{\delta}(u_1(s)) - f_{\delta}(u_2(s))  \right\|_{ L^{2}\left(\Omega\right)} \mathrm{d} s \right\|_{L^{\infty}\left([0, T] \right)}
    \\ \le &  \left\|\int_{0}^{t}  \left( 2\lambda\delta^{-3} + 2\lambda\delta^{-4}|\Omega|\right)\left\|u_1-u_2\right\|_{ L^{2}\left(\Omega\right)} \mathrm{d} s \right\|_{L^{\infty}\left([0, T] \right)}
    \\ = &  \sup_{t\in [0, T]}ess \int_{0}^{t}  \left( 2\lambda\delta^{-3} + 2\lambda\delta^{-4}|\Omega|\right)\left\|u_1-u_2\right\|_{ L^{2}\left(\Omega\right)} \mathrm{d} s
    \\ \le &  T\sup_{t\in [0, T]}ess    \left( 2\lambda\delta^{-3} + 2\lambda\delta^{-4}|\Omega|\right)\left\|u_1-u_2\right\|_{ L^{2}\left(\Omega\right)}
    \\ = &  T  \left( 2\lambda\delta^{-3} + 2\lambda\delta^{-4}|\Omega|\right)\left\|u_1-u_2\right\|_{L^{\infty}\left([0, T] ; L^{2}(\Omega)\right)},
    \end{aligned}
    \end{equation*}
    we take $T_0=\min\left\{\left(|\Omega|\lambda\delta^{-2}
    \left(1+\delta^{-1}|\Omega|\right)^{-2}\right)^{-1},
    \left( 2\lambda\delta^{-3} + 2\lambda\delta^{-4}|\Omega|\right)^{-1}\right\}$, when $T<T_0$, $B$ is a contraction operator on $Y_{M}$. According to contraction mapping principle, equation (\ref{014}) has a unique mild solution
    \begin{equation*}
    \begin{aligned}
     u \in C\left([0, T];  L^{2}(\Omega)\right).
    \end{aligned}
    \end{equation*}
    Further using Lemma \ref{classic exists 2}, we can obtain that $u$ is a classical solution of equation (\ref{014}), that is,
    \begin{equation*}
    \begin{aligned}
        u \in C\left([0, T]; H^{2}(\Omega) \cap H_{0}^{1}(\Omega)\right) \cap C^{1}\left([0, T];  L^{2}(\Omega)\right).
    \end{aligned}
    \end{equation*}

    By sobolev embedding theorem, when $1\le N\le 3$, $H^2(\Omega)$ is continuously embedded into $L^{\infty}(\Omega)$. Therefore, for $\delta >0$, $\exists T_{1}>0$, such that when $0<t<T_{1}$,
    \begin{equation*}
    \left\| u(t)-u_0\right\|_{ L^{\infty}\left(\Omega\right)}
    \le \delta.
    \end{equation*}

    Thus $u(t)\le 1-\delta ,\quad \forall t\in [0, T_1]$. Let $\widetilde{T}=\min\{T_0, T_1\}$, then when $T<\widetilde{T}$, $f_\delta(u)\equiv f(u)$. At this time, equation (\ref{014}) is equivalent to equation (\ref{1.4}), then equation (\ref{1.4}) has a unique classical solution
    \begin{equation*}
        u \in C\left([0, T]; H^{2}(\Omega) \cap H_{0}^{1}(\Omega)\right) \cap C^{1}\left([0, T];  L^{2}(\Omega)\right).
    \end{equation*}
\end{proof}

Next we prove Theorem \ref{local existence}.

\begin{proof}[Proof of Theorem \ref{local existence}]
    The first half of the theorem has been proved in proposition \ref{lem 01}.\\
    We consider a family of closed sets $\{I_n\}$ on which equation (\ref{1.4}) has a solution
    $$u_n\in C\left(I_n; H^{2}(\Omega) \cap H_{0}^{1}(\Omega)\right) \cap C^{1}\left(I_n;  L^{2}(\Omega)\right).$$
And we define the following partial order relation $"\preceq"$:\\
If $I_p\subset I_q\subset [0, +\infty)$, and $\forall t\in I_p, u_p=u_q$, then $I_p\preceq I_q$.\\
Thus we get a non-empty partial order set $P=(I_n\subset[0, +\infty);\preceq)$. By Zorn's lemma, there exists a maximum element in $P$, denoted by $I_0$.

Assume $I_0=[0, T^*],\quad 0<T^*<\infty$, then the solution satisfies
$$u^{*}\in C\left(I_0; H^{2}(\Omega) \cap H_{0}^{1}(\Omega)\right) \cap C^{1}\left(I_0,  L^{2}(\Omega)\right),$$
and $\displaystyle 0<\max_{x\in \Omega}u^{*}{(T^*, x )}<1$.

Therefore, for $\displaystyle \delta^{'}=(1-\max_{x\in \Omega}u^{*}{(T^*, x)})/2$, we have $u^{*}(T^*)<1-2\delta^{'}$ and consider the equation
\begin{equation}
\left\{\begin{array}{lll}
 v_{t}=\Delta v+f(v),\quad &(t,x) \in  (0, \infty)\times\Omega,
\\    v=0 ,\quad&(t,x) \in (0, \infty)\times \partial \Omega, 	
\\  v(0, x)=u^{*}(T^*, x),\quad & x \in \Omega.
\end{array}\quad \right. \label{019}
\end{equation}
From the proof of the local existence of the solution, we obtain that $\exists\  T_2=T_2(\delta^{'}, \lambda, \Omega)>0$, equation (\ref{019}) has a solution
$$v_1\in C\left([0, T_2]; H^{2}(\Omega) \cap H_{0}^{1}(\Omega)\right) \cap C^{1}\left([0, T_2];  L^{2}(\Omega)\right).$$
Thus $I_0\preceq[0, T^*+T_2]$, which contradicts the assumption that $I_0$ is a maximum element, then $I_0=[0, T^*), T^*>0$.

We also assume that $T^*<\infty$, $u$ is a solution of equation (\ref{1.1}) and satisfies
\begin{equation}
\left\{\begin{array}{l}u \in C\left(I_n; H^{2}(\Omega) \cap H_{0}^{1}(\Omega)\right) \cap C^{1}\left(I_n;  L^{2}(\Omega)\right),
\\ \displaystyle\sup _{(t,x) \in [0, T^*) \times\Omega} u(t, x) \le 1-\sigma,\end{array}\right.
\end{equation}
where $\sigma \in(0, 1)$,
then $\exists \ \{t_n\}\nearrow T^*,\quad n\to \infty$, such that $\displaystyle \sup_{x\in \Omega} u(t_n, x)\le 1-\sigma$.

We consider the equation
\begin{equation}
\left\{\begin{array}{lll}
 v_{t}=\Delta v+f(v),\quad & (t,x) \in (0, \infty) \times\Omega,
\\    v=0 ,\quad& (t,x) \in  (0, \infty) \times \partial \Omega 	,
\\  v(0, x)=u(t_n, x) ,\quad& x \in \Omega.
\end{array}\quad \right. \label{020}
\end{equation}

From the proof of the local existence of the solution, we obtain that $\exists\  T_3=T_3(\sigma, \lambda, \Omega)>0$, equation (\ref{020}) has a solution
$$v_2\in C\left([0, T_3]; H^{2}(\Omega) \cap H_{0}^{1}(\Omega)\right) \cap C^{1}\left([0, T_3];  L^{2}(\Omega)\right),\quad \forall n\in N. $$
Since $\displaystyle \lim_{n\to \infty}t_n=T^*$, we have that $\exists\  N>0$, when $n>N$, $T^*-t_n<T_3 $.

Let
\begin{equation}
\widetilde{u}(t)=\left\{\begin{aligned}
& u(t),\qquad &[0, t_m],
\\ & v(t-t_m),\qquad &[t_m, t_m+T_3],
\end{aligned}\quad \right. \label{021}
\end{equation}
where $m=N+1$.\\
Next we will prove
\begin{equation}
	\widetilde{u}(t)=Q(t) u_{0}+\int_{0}^{t} Q(t-s) f(\widetilde{u}(s)) \mathrm{d} s,\text{ }\forall t\in[0, t_m+T_3].\label{022}
\end{equation}
Note that
\begin{equation}
u(t)=Q(t) u_{0}+\int_{0}^{t} Q(t-s) f(u(s)) \mathrm{d} s,\text{ }\forall t\in[0, t_m],
\label{023}
\end{equation}
and
\begin{equation}
v(t)=Q(t) u(t_m, x)+\int_{0}^{t} Q(t-s) f(v(s)) \mathrm{d}s ,\text{ }\forall  t\in[0, T_3].\label{024}
\end{equation}
By (\ref{021}), we have
\begin{equation}
	\widetilde{u}(t)=Q(t) u_{0}+\int_{0}^{t} Q(t-s) f(\widetilde{u}(s)) \mathrm{d} s,\text{ }\forall  t\in[0, t_m].\label{025}
\end{equation}
For $t\in [0, T_3]$, denote $t'=t+t_m\in [t_m, t_m+T_3]$,
\begin{equation}
	\begin{aligned}
	& \widetilde{u}\left(t'\right)=\widetilde{u}\left(t+t_{m}\right)=v(t)=Q(t) u\left(t_{m}, x\right)+\int_{0}^{t} Q(t-s) f(v(s)) d s
\\=& Q(t)\left(Q\left(t_{m}\right) u_{0}+\int_{0}^{t_{m}} Q\left(t_{m}-s\right) f(u(s)) d s\right)+\int_{0}^{t} Q(t-s) f(v(s)) d s \\=& Q\left(t+t_{m}\right) u_{0}+\int_{0}^{t_{m}} Q\left(t+t_{m}-s\right) f(u(s)) d s +\int_{t_{m}}^{t+t_{m}} Q\left(t+t_{m}-s\right) f\left(v\left(s-t_{m}\right)\right) d s
\\=& Q\left(t+t_{m}\right) u_{0}+\int_{0}^{t_{m}} Q\left(t+t_{m}-s\right) f(\widetilde{u}(s)) d s +\int_{t_{m}}^{t+t_{m}} Q\left(t+t_{m}-s\right) f(\widetilde{u}(s)) d s
\\=& Q\left(t+t_{m}\right) u_{0}+\int_{0}^{t+t_{m}} Q\left(t+t_{m}-s\right) f(\widetilde{u}(s)) d s
\\=& Q\left(t'\right) u_{0}+\int_{0}^{t'} Q\left(t'-s\right) f(\widetilde{u}(s)) d s.
	\end{aligned}\label{026}
\end{equation}
Therefore, $\widetilde{u}$ is a solution of equation (\ref{1.4}) and
$$\widetilde{u} \in C\left([0, t_m+T_3]; H^{2}(\Omega) \cap H_{0}^{1}(\Omega)\right) \cap C^{1}\left([0, t_m+T_3];  L^{2}(\Omega)\right),\quad\forall m\in N. $$
Thus $I_0\preceq [0, t_m+T_3]$, which contradicts the assumption that $I_0$ is a maximum element.\\
Theorem \ref{local existence} is proved.
\end{proof}

\section{Global existence}

\qquad In this section we prove Theorem \ref{global existence}. We will use the local existence results to obtain the global existence of the solution under certain conditions through a series of estimates, and illustrate that the solution converges to the minimal steady-state solution at an exponential rate with $t\to \infty$.

\begin{lemma}[\cite{lia}]
    If $\Omega$ is a strictly star-shaped region in $\mathbb{R}^{N}$ ($N \geq 1$), then there exists two constants $0<\lambda_{*} \le \lambda^{*}<\infty$, and when $\lambda<\lambda_{*}$, equation (\ref{1.4}) has at least one minimal solution $w_{\lambda}$ that satisfies $\forall x \in \Omega$, $w_{\lambda}(x)<1$. We call this solution the minimal steady-state solution of equation (\ref{1.1}). When $\lambda>\lambda^{*}$, equation (\ref{1.4}) has no solution. In particular, when $N=1$, $\lambda_{*} = \lambda^{*}$.
\label{sou 1}
\end{lemma}

\begin{proof}[Proof of Theorem \ref{global existence}]
For $w_\lambda<1$, we choose an appropriate $\delta\in\left(0, \frac{1}{2}\right)$ such that
\begin{equation}
w_{\lambda}(x) \le 1-2 \delta ,\quad u_{0}(x) \le 1-2 \delta  ,\text{ } \forall x \in \Omega.\label{033}
\end{equation}
Let $ w(t, x)=u(t, x)-w_{\lambda}(x)$, we have:
\begin{equation}
\left\{\begin{array}{lll}
 -\Delta w+w_{t}=f_{\delta}(u)-f_{\delta}
\left(w_{\lambda}\right),\quad & (t,x) \in(0, T)\times\Omega ,
\\  w(t, x)=0,\quad &(t,x) \in (0, T)\times \partial \Omega,
\\  w(0, x)=u_{0}(x)-w_{\lambda}(x),\quad & x \in\Omega,
\end{array}\quad \right. \label{034}
\end{equation}
Where $T$ is defined in proposition \ref{lem 01}.

We multiply the $\eqref{034}$ by $w_t$ and $w$ respectively, add them together, and integrate over $\Omega$:
\begin{equation}
\begin{aligned}
&\int_{\Omega}-w_t\Delta w+w_{t}^2-w\Delta w+ww_{t} d x
\\=&\frac{1}{2}\frac{d}{d t} \int_{\Omega}\left(|w|^{2}+|\nabla w|^{2}\right) d x+\int_{\Omega}\left(\left|w_{t}\right|^{2}+|\nabla w|^{2}\right) d x \\=&  \int_{\Omega}\left[F_{\delta}(u)-F_{\delta}\left(w_{\lambda}\right)\right] w_{t} d x+\int_{\Omega}\left[F_{\delta}(u)-F_{\delta}\left(w_{\lambda}\right)\right] w d x
\\ \le & C \lambda \delta^{-3} \int_{\Omega}|w|\left(|w|+\left|w_{t}\right|\right) d x+C \lambda \delta^{-4} \int_{\Omega}|w| d x \cdot \int_{\Omega}\left[\left|w_{t}\right|+|w|\right] d x \\ \le & C \lambda \delta^{-4} \int_{\Omega}\left(\left|w_{t}\right|^{2}+|\nabla w|^{2}\right) d x.
\end{aligned}\label{035}
\end{equation}

Next we estimate $w_t$.

Differentiating the equation (\ref{034}) with respect to $t$:
\begin{equation}
\left\{\begin{array}{lll}
 -\Delta w_t+w_{tt}=\partial_{t}\left(f_{\delta}(u)
-f_{\delta}\left(w_{\lambda}\right)\right),\quad & (t,x) \in(0, T)\times \Omega,
\\  w_{t}(t, x)=0 ,\quad& (t,x) \in(0, T)\times \partial \Omega,
\\  w_{t}(0, x)=\Delta\left(u_{0}(x)-w_{\lambda}(x)\right)
+f_{\delta}(u_0(x))-f_{\delta}\left(w_{\lambda}(x)
\right) ,\quad & \text { in } \Omega,
\end{array}\quad \right. \label{036}
\end{equation}
where
\begin{equation*}
\begin{aligned}
&\quad\ \ \partial_{t}\left(f_{\delta}(u)-f_{\delta}\left(w_{\lambda}\right)\right)\\
&=\frac{-2 \lambda\left(g_{\delta}(u)+g_{\delta}\left(w_{\lambda}\right)\right)}{\left(1+\int_{\Omega} g_{\delta}(u) d x\right)^{3}} \cdot\left(g_{\delta}(u)-g_{\delta}\left(w_{\lambda}\right)\right) \int_{\Omega} g_{\delta}^{\prime}(u) w_{t} d x
\\  &+\frac{2 \lambda g_{\delta}(u) g_{\delta}^{\prime}(u)}{\left(1+\int_{\Omega} g_{\delta}(u) d x\right)^{2}} \cdot v-\frac{2 \lambda\left(g_{\delta}\left(w_{\lambda}\right)\right)^{2}}{\left(1+\int_{\Omega} g_{\delta}\left(w_{\lambda}\right) d x\right)^{2}\left(1+\int_{\Omega} g_{\delta}(u) d x\right)} \cdot \int_{\Omega} g_{\delta}^{\prime}(u) w_{t} d x
\\  &+\frac{2 \lambda\left(g_{\delta}\left(w_{\lambda}\right)\right)^{2}\left[2+\int_{\Omega}\left(g_{\delta}(u)+g_{\delta}\left(w_{\lambda}\right)\right) d x\right]}{\left(1+\int_{\Omega} g_{\delta}\left(w_{\lambda}\right) d x\right)^{2}\left(1+\int_{\Omega} g_{\delta}(u) d x\right)^{3}} \cdot \int_{\Omega}\left(g_{\delta}(u)-g_{\delta}\left(w_{\lambda}\right)\right) d x \int_{\Omega} g_{\delta}^{\prime}(u) w_{t} d x,
\end{aligned}
\end{equation*}
thus
\begin{equation}
\left|\partial_{t}\left(f_{\delta}(u)-f_{\delta}\left(w_{\lambda}\right)\right)\right| \le C \lambda \delta^{-6}\left(\|w_{t}\|_{L^{2}(\Omega)} \cdot|w|+|w_{t}|+\|w_{t}\|_{L^{2}(\Omega)}+\|w_{t}\|_{L^{2}(\Omega)}
\|w\|_{L^{2}(\Omega)}\right).	\label{037}
\end{equation}
Multiplying (\ref{036}) by $w_t$ and integrate over $\Omega$, we have
\begin{equation}
\begin{aligned}
&\int_{\Omega}-w_{t}\Delta w_{t}+w_{t}w_{tt} d x
=\frac{1}{2}\frac{d}{d t} \int_{\Omega}|w_{t}|^{2} d x+\int_{\Omega}|\nabla w_{t}|^{2} d x \\=&  \int_{\Omega}\partial_{t}\left(f_{\delta}(u)-f_{\delta}\left(w_{\lambda} \right)\right) w_{t} d x  \le  C \lambda \delta^{-6} \int_{\Omega}\left|w_{t}\right|^{2} d x.
\end{aligned}\label{038}
\end{equation}
Adding \eqref{035} and \eqref{038} yields:
\begin{equation}
\begin{aligned}
&\frac{1}{2}\frac{d}{d t} \int_{\Omega}\left(\left|w\right|^{2}+|\nabla w|^{2}+| w_t|^{2}\right) d x+\int_{\Omega}\left(\left|w_{t}\right|^{2}+|\nabla w|^{2}+|\nabla w_t|^{2}\right) d x
\\ \le &C \lambda \delta^{-6} \int_{\Omega}\left(\left|w_{t}\right|^{2}+|\nabla w|^{2}+| w_t|^{2}\right) d x.
\end{aligned} \label{039}
\end{equation}
Now taking $\lambda \le \frac{\delta^{6}}{3 C}$, we have:
\begin{equation*}
C \lambda \delta^{-6} \int_{\Omega}\left(\left|w_{t}\right|^{2}+|\nabla w|^{2}+| w_t|^{2}\right) d x \le  \frac{1}{3}  \int_{\Omega}\left(\left|w_{t}\right|^{2}+|\nabla w|^{2}+| w_t|^{2}\right) d x,
\end{equation*}
by \eqref{039},
\begin{equation}
\begin{aligned}
&\frac{1}{2}\frac{d}{d t} \int_{\Omega}\left(\left|w\right|^{2}+|\nabla w|^{2}+| w_t|^{2}\right) d x+\frac{1}{3}\int_{\Omega}\left(\left|w_{t}\right|^{2}+|\nabla w|^{2}\right) d x
\\ \le &- \int_{\Omega}|\nabla w_{t}|^{2} d x \le 0.
\end{aligned} \label{040}
\end{equation}
by Theorem (Poincaré Inequality):
\begin{equation*}
\begin{aligned}
&\frac{d}{d t} \int_{\Omega}\left(\left|w\right|^{2}+|\nabla w|^{2}+| w_t|^{2}\right) d x+C_2\int_{\Omega}\left(\left|w\right|^{2}+|\nabla w|^{2}+| w_t|^{2}\right) d x \le 0,
\end{aligned}
\end{equation*}
then
\begin{equation*}
\frac{d}{d t}\left( e^{C_2t}\int_{\Omega}\left(\left|w\right|^{2}+|\nabla w|^{2}+| w_t|^{2}\right)  d x\right) \le 0,
\end{equation*}
thus
\begin{equation}
\begin{aligned}
&\int_{\Omega}\left(\left|w\right|^{2}+|\nabla w|^{2}+| w_t|^{2}\right)  d x
\\ \le & e^{-C_2t}\int_{\Omega}\left(\left|w(0, x)\right|^{2}+|\nabla w(0, x)|^{2}+| w_t(0, x)|^{2}\right)  d x
\\=&\left\|u_{0}-w_{\lambda}\right\|_{H^{1}}^2e^{-C_2t}+\left\|w_t(0, x)\right\|_{L^{2}}^2e^{-C_2t}.
\end{aligned} \label{041}
\end{equation}
Note that:
\begin{equation*}
\left\|w_t(0, x)\right\|_{L^{2}(\Omega)}^2=\left\|\Delta\left(u_{0}(x)
-w_{\lambda}(x)\right)+f_{\delta}(u_0(x))-f_{\delta}
\left(w_{\lambda}(x)\right)\right\|_{L^{2}(\Omega)}^2\le C_3\left\|u_{0}-w_{\lambda}\right\|_{H^{2}(\Omega)}^2,
\end{equation*}
By $\Delta w=w_t-\left(f_{\delta}(u)-f_{\delta}
\left(w_{\lambda}\right)\right)$ and \eqref{041}, we have
\begin{equation}
\left\|w\right\|_{H^{2}(\Omega)}^2
\le C_2\left\|u_{0}-w_{\lambda}\right\|_{H^{2}(\Omega)}^2 e^{-C_1t},
\label{042}
\end{equation}
then
\begin{equation}
\left\|\left\|w\right\|_{H^{2}(\Omega)}^2\right\|_{L^{\infty}([0, T])}
\le C_2\left\|u_{0}-w_{\lambda}\right\|_{H^{2}(\Omega)}^2 . \label{043}
\end{equation}
By Theorem Sobelve,
\begin{equation}
\qquad\left\|u-w_{\lambda}\right\|_{L^{\infty}([0, T] \times \Omega)} \le C\left\|u_{0}-w_{\lambda}\right\|_{H^{2}(\Omega)}.\label{044}
\end{equation}
Taking $\theta=\min \left\{\frac{\delta}{C}, 1\right\}$ and $ \lambda_{0}= \frac{\delta^{6}}{3 C} $, we have $u(t, x) \le 1-\delta ,\quad \forall (t, x) \in[0, T] \times \Omega$.

This means that $\|u\|_{L^{\infty}([0, T] \times \Omega)} \le 1-\delta$. At this time, $f_\delta(u)\equiv f(u)$, then $u$ is the solution of equation (\ref{1.1}) on $[0, T]$.

According to Theorem \ref{local existence}, we extend the solution $u$ from $T$  to the maximum existence interval $\widetilde{T}$, and $\widetilde{T}$ depends on $\lambda$, $\Omega$ and $\delta $.

Let $v(t, x)=u(t+T, x)-w_\lambda(x)$, consider the equation:
\begin{equation}
\left\{\begin{array}{lll}
 -\Delta v+v_{t}=f_{\delta}(v)-f_{\delta}\left(w_{\lambda}\right) ,\quad& (t,x) \in (0,  \widetilde{T}-T)\times\Omega,
\\  v(t, x)=0 ,\quad& (t,x) \in (0, \widetilde{T}-T)\times\partial \Omega ,
\\  v(0, x)=u(T, x)-w_{\lambda}(x) ,\quad& x\in \Omega.
\end{array}\quad \right. \label{045}
\end{equation}
Repeating the above process, we can get:
\begin{equation}
\qquad\left\|u(t+T, x)-w_\lambda(x)\right\|_{L^{\infty}([0, \widetilde{T}-T] \times \Omega)} \le C\left\|u(T, x)-w_{\lambda}\right\|_{H^{2}(\Omega)}
\le CC_2\left\|u_{0}-w_{\lambda}\right\|_{H^{2}(\Omega)}^2.\label{046}
\end{equation}
Taking $\theta=\min \left\{\frac{\delta}{C}, 1, \frac{\delta}{CC_2}, \frac{1}{C_2}\right\}$ and $\lambda_{0}=\frac{\delta^6}{3C}$, we have
$$\|u(t+T, x)\|_{L^{\infty}([0, \widetilde{T}-T] \times \Omega)} \le 1-\delta,$$
then
$$\|u\|_{L^{\infty}([0, \widetilde{T}] \times \Omega)} \le 1-\delta,$$
therefore, $\widetilde{T}=\infty$. Theorem \ref{global existence} is proved.
\end{proof}

\section{Asymptotic behavior of solutions}

\qquad In this section, we will use gradient system theory, analyticity theory and Lojasiewicz-Simon method to study the convergence and convergence rate of the global solution to the steady-state solution.

\subsection{Gradient system}

 \begin{definition}[See Definition 6.3.1 in \cite{ref8}]
    Let $H$ be a complete metric space, $S(t)$ be a $C_0$-semigroup on $H$, and $V:H\to \mathbb{R}$ be a continuous functional. If the following conditions hold:\\
    (1) For every $x\in H$, $V(S(t)x)$ is monotonically decreasing in $t$;\\
    (2) $V(x)\ge C$ for all $x\in H$, where $C$ is a constant.\\
    Then $V$ is called a Lyapunov functional.
\end{definition}

\begin{definition}[See Definition 6.3.2 in \cite{ref8}]
    Let $H$ be a complete metric space, $S(t)$ be a $C_0$-semigroup on $H$, and $V$ be a Lyapunov functional. If the following conditions hold:\\
    (1) For every $x\in H$, there exists $t_0>0$ such that $\bigcup_{t\ge t_0}S(t)x$ is relatively compact in $H$;\\
    (2) If $V(S(t)x)=V(x)$ for all $t>0$, then $x$ is a fixed point of $S(t)$.\\
    Then $(H,S(t),V)$ is called a gradient system.
\end{definition}

To construct a gradient system, we will find a Lyapunov functional and a $C_0$-semigroup for the equation and prove the compactness of the trajectories.

Multiply both sides of equation (\ref{1.1}) by $u_t$ and integrate over $\Omega$:
$$\int_{\Omega}u_{t}^{2}\,\mathrm{d}x-\int_{\Omega}u_t\Delta u\,\mathrm{d}x = \int_{\Omega}\frac{\lambda u_t}{(1-u)^2(1+\int_{\Omega}\frac{1}{1-u}\,\mathrm{d}x)^2}\,\mathrm{d}x.$$
By Green's formula,
$$\int_{\Omega}u_t\Delta u\,\mathrm{d}x = -\int_{\Omega}\nabla u_t\nabla u\,\mathrm{d}x +\int_{\partial \Omega}u_t\frac{\partial u}{\partial n}\,\mathrm{d}S.$$
Combining with the Dirichlet boundary conditions, we have $\int_{\Omega}u_t\Delta u\,\mathrm{d}x = -\int_{\Omega}\nabla u_t\nabla u\,\mathrm{d}x$. Therefore,
$$\int_{\Omega}u_{t}^2\,\mathrm{d}x+\frac{\mathrm{d}}{\mathrm{d}t}\left(\frac{1}{2}\int_{\Omega}|\nabla u|^2\,\mathrm{d}x\right)=\frac{\mathrm{d}}{\mathrm{d}t}\left(-\lambda\left(1+\int_{\Omega}\frac{1}{1-u}\,\mathrm{d}x\right)^{-1}\right),$$
which implies
$$\frac{\mathrm{d}}{\mathrm{d}t}\left(\frac{1}{2}\int_{\Omega}|\nabla u|^2\,\mathrm{d}x+\frac{\lambda}{1+\int_{\Omega}\frac{1}{1-u}\,\mathrm{d}x}\right)=-\int_{\Omega}u_{t}^2\,\mathrm{d}x.$$
We define the energy functional for equation (\ref{1.1}) as
\begin{equation}
  E(t)(u)=\frac{1}{2}\int_{\Omega}|\nabla u|^2\,\mathrm{d}x+\frac{\lambda}{1+\int_{\Omega}\frac{1}{1-u}\,\mathrm{d}x}, \label{2.1}
\end{equation}
and
\begin{equation}
  \frac{\,\mathrm{d}E}{\,\mathrm{d}t}=-\int_{\Omega}u_{t}^2\,\mathrm{d}x\le 0.\label{2.2}
\end{equation}
Note that $E:H^2\cap H_0^1(\Omega)\to\mathbb{R}$ is continuous and
$$E(t)(u)\ge 0,\quad \forall u\in H^2\cap H_0^1(\Omega).$$
Therefore, $E:H^2\cap H_0^1(\Omega)\to\mathbb{R}$ is a Lyapunov functional.

For every $\lambda>0$ and $u_0\in H^2\cap H_0^1(\Omega)$, by assumption (\ref{u}), the global solution satisfies $u\in C([0,+\infty),H^2\cap H_0^1(\Omega))\cap C^1([0,+\infty),L^2(\Omega))$. Thus, we can define a nonlinear $C_0$-semigroup on $H^2\cap H_0^1(\Omega)$ as:
$$S(t):u_0\mapsto u(t)=S(t)u_0.$$

\begin{lemma}[See Lemma 6.2.1 in \cite{ref8}]\label{lemma2.1}
    Let $y(t)$ and $h(t)$ be nonnegative continuous functions defined on $[0,T]$ ($0<T\le +\infty$). If the following hold:
    $$\frac{\,\mathrm{d} y}{\,\mathrm{d} t}\le A_1y^2+A_2+h(t),$$
    $$\int_{0}^{T}y(t)\,\mathrm{d}t\le A_3,\quad \int_{0}^{T}h(t)\,\mathrm{d}t\le A_4,$$
    where $A_i\ge 0$ ($i=1,\dots,4$) are constants. Then for any $r\in(0,T)$:
    $$y(t+r)\le \left(\frac{A_3}{r}+A_2r+A_4\right)e^{A_1A_3},\quad t\in [0,T-r).$$
    In particular, if $T=+\infty$:
    $$\lim_{t \to \infty}y(t)=0.$$
\end{lemma}

\begin{proposition}
  $(H^2\cap H_0^1(\Omega),S(t),E)$ forms a gradient system.
\end{proposition}

\begin{proof}
Integrate both sides of (\ref{2.2}) with respect to $t$:
\begin{equation}
    E(u(t))+\int_{0}^{t}\|u_t\|_{L^2}^{2}\,\mathrm{d}\tau=E(u_0).\label{2.3}
\end{equation}
If there exists $t_0>0$ such that $E(u(t_0))=E(S(t_0)u_0)=E(u_0)$, then $\int_{0}^{t_0}\|u_t\|_{L^2}^{2}\,\mathrm{d}\tau=0$, which implies $u_t=0$ for all $t\in [0,t_0]$. Hence, $u_0$ is a fixed point of $S(t)$.

Next, we prove there exists $t_0>0$ such that $\bigcup_{t\ge 0}S(t)u_0$ is relatively compact in $H^2\cap H_0^1(\Omega)$.

Differentiate both sides of (\ref{1.1}) with respect to $t$:
\begin{equation}
  u_{tt}-\Delta u_t=\frac{2\lambda u_t}{(1-u)^3(1+\int_{\Omega}\frac{1}{1-u}\,\mathrm{d}x)^2}-\frac{2\lambda\int_{\Omega}\frac{u_t}{(1-u)^2}\,\mathrm{d}x}{(1-u)^2(1+\int_{\Omega}\frac{1}{1-u}\,\mathrm{d}x)^3}.\label{2.4}
\end{equation}
Multiply both sides of (\ref{2.4}) by $u_t$ and integrate over $\Omega$:
\begin{equation}
  \begin{aligned}
    &\quad\frac{1}{2}\frac{\mathrm{d}}{\mathrm{d}t}\|u_t\|_{L^2}^2+\|\nabla u_t\|_{L^2}^2\\
    &= 2\int_{\Omega}\frac{\lambda u_{t}^{2}}{(1-u)^3(1+\int_{\Omega}\frac{1}{1-u}\,\mathrm{d}x)^2}\,\mathrm{d}x-2\int_{\Omega}\frac{\lambda u_t\int_{\Omega}\frac{u_t}{(1-u)^2}\,\mathrm{d}x}{(1-u)^2(1+\int_{\Omega}\frac{1}{1-u}\,\mathrm{d}x)^3}\,\mathrm{d}x\\
    &= \frac{2\lambda}{(1+\int_{\Omega}\frac{1}{1-u}\,\mathrm{d}x)^2}\int_{\Omega}\frac{u_{t}^{2}}{(1-u)^3}\,\mathrm{d}x-\frac{2\lambda(\int_{\Omega}\frac{u_t}{(1-u)^2}\,\mathrm{d}x)^2}{(1+\int_{\Omega}\frac{1}{1-u}\,\mathrm{d}x)^3}\\
    &\le \frac{2\lambda}{(1+\int_{\Omega}\frac{1}{1-u}\,\mathrm{d}x)^2}\int_{\Omega}\frac{u_{t}^{2}}{(1-u)^3}\,\mathrm{d}x\\
    &\le C_1\|u_t\|_{L^2}^{2},
  \end{aligned}\label{2.5}
\end{equation}
where $C_1=C_1(\lambda,u_0,\Omega)>0$. The last inequality holds because $u$ satisfies assumption (\ref{u}), so $\|\frac{1}{1-u}\|_{L^\infty}\le M$ for all $t\ge 0$.

Integrate (\ref{2.5}) with respect to $t$:
$$\frac{1}{2}(\|u_t(t)\|_{L^2}^2-\|u_t(0)\|_{L^2}^2)+\int_{0}^{t}\|\nabla u_t\|_{L^2}^2\,\mathrm{d}\tau\le C_1\int_{0}^{t}\|u_t\|_{L^2}^2\,\mathrm{d}\tau.$$
From (\ref{2.3}),
\begin{equation}
    \int_{0}^{t}\|u_t\|_{L^2}^{2}\,\mathrm{d}\tau=E(u_0)-E(u(t))\le E(u_0)\le C_2,\quad \forall t\ge 0.\label{2.6}
\end{equation}
Thus,
\begin{equation}
    \frac{1}{2}(\|u_t(t)\|_{L^2}^2-\|u_t(0)\|_{L^2}^2)+\int_{0}^{t}\|\nabla u_t\|_{L^2}^2\,\mathrm{d}\tau\le C_1C_2,\quad \forall t\ge 0.\label{2.7}
\end{equation}
Apply Young's Inequality to the right-hand side of (\ref{2.5}):
$$\frac{1}{2}\frac{\mathrm{d}}{\mathrm{d}t}\|u_t\|_{L^2}^2+\|\nabla u_t\|_{L^2}^2\le C_1\|u_t\|_{L^2}^{2}\le C_1\left(\frac{1}{2}\|u_t\|_{L^2}^{4}+\frac{1}{2}\right),$$
$$\frac{\mathrm{d}}{\mathrm{d}t}\|u_t\|_{L^2}^2\le C_1\|u_t\|_{L^2}^{4}+C_1.$$
By Lemma \ref{lemma2.1},
$$\lim_{t \to \infty}\|u_t\|_{L^2}^2=0,$$
so
$$\|u_t\|_{L^2}\le C_3,\quad \forall t\ge 0.$$
Substitute into (\ref{2.7}):
\begin{equation}
    \int_{0}^{t}\|\nabla u_t\|_{L^2}^2\,\mathrm{d}\tau\le C,\quad \forall t\ge 0,\label{2.8}
\end{equation}
where $C>0$ depends on $\lambda,u_0,\Omega$.

Multiply both sides of (\ref{2.4}) by $-\Delta u_t$ and integrate over $\Omega$:
\begin{align*}
    &\quad\frac{1}{2}\frac{\mathrm{d}}{\mathrm{d}t}\|\nabla u_t\|_{L^2}^2+\|\Delta u_t\|_{L^2}^2\\
    &= \frac{2\lambda}{(1+\int_{\Omega}\frac{1}{1-u}\,\mathrm{d}x)^2}\int_{\Omega}\frac{u_{t}}{(1-u)^3}(-\Delta u_t)\,\mathrm{d}x-\frac{2\lambda\int_{\Omega}\frac{u_t}{(1-u)^2}\,\mathrm{d}x}{(1+\int_{\Omega}\frac{1}{1-u}\,\mathrm{d}x)^3}\int_{\Omega}\frac{-\Delta u_t}{(1-u)^2}\,\mathrm{d}x\\
    &\le C_4\|u_t\|_{L^2}\|\Delta u_t\|_{L^2}+C_5\|u_t\|_{L^2}\|\Delta u_t\|_{L^2}\\
    &\le C\|u_t\|_{L^2}\|\Delta u_t\|_{L^2}\\
    &\le \frac{1}{2}\|\Delta u_t\|_{L^2}+\frac{C^2}{2}\|u_t\|_{L^2}^2.
\end{align*}
Hence,
\begin{equation}
    \frac{\mathrm{d}}{\mathrm{d}t}\|\nabla u_t\|_{L^2}^2+\|\Delta u_t\|_{L^2}^2\le C^2\|u_t\|_{L^2}^2.\label{2.9}
\end{equation}
Multiply both sides of (\ref{2.9}) by $t$ and integrate with respect to $t$:
$$t\|\nabla u_t\|_{L^2}^2-\int_{0}^{t}\|\nabla u_t\|_{L^2}^2\,\mathrm{d}\tau+\int_{0}^{t}\tau\|\Delta u_t\|_{L^2}^2\,\mathrm{d}\tau\le C^2\int_{0}^{t}\tau\|u_t\|_{L^2}^2\,\mathrm{d}\tau,\quad \forall t>0.$$
Using (\ref{2.6}) and (\ref{2.8}),
$$t\|\nabla u_t\|_{L^2}^2+\int_{0}^{t}\tau\|\Delta u_t\|_{L^2}^2\,\mathrm{d}\tau\le C+C^2t\int_{0}^{t}\|u_t\|_{L^2}^2\,\mathrm{d}\tau\le C+C^3t.$$
Thus, for $t\ge\delta>0$:
$$\|\nabla u_t\|_{L^2}^2\le \frac{C}{t}+C^3\le C_{\delta}=\frac{C}{\delta}+C^3,$$
and consequently,
$$\|u\|_{H^3}\le C\left(\|u_t\|_{H^1}+\left\|\frac{\lambda}{(1-u)^2(1+\alpha\int_{\Omega}\frac{1}{1-u}\,\mathrm{d}x)^2}\right\|_{H^1}\right)\le {C_{\delta}}',\quad \forall t\ge\delta>0.$$
This shows that for every $u_0\in H^2\cap H_0^1(\Omega)$, there exists $\delta>0$ such that $\bigcup_{t\ge \delta}S(t)u_0$ is relatively compact in $H^2\cap H_0^1$.

Therefore, $(H^2\cap H_0^1(\Omega),S(t),E)$ forms a gradient system.
\end{proof}

\subsection{Analyticity}

\qquad Analyticity is one of the conditions for the Lojasiewicz-Simon Inequality to hold. This section will prove that the energy functional $E$ is analytic near equilibrium points.

\begin{proposition}
    Let $\psi$ be an equilibrium point satisfying (\ref{1.5}), and $V$ be an open neighborhood of $\psi$ in $H^2\cap H_0^1(\Omega)$. Then the energy functional $E$ is analytic on $V$.
\end{proposition}

\begin{proof}
    Define the function $f:\Omega\times D\to\mathbb{R}$ with $D=(0,1)$ and expression:
    $$f(x,u)=\frac{1}{1-u}.$$
    For every $D'\subset\subset D$, there exists $r<\inf |1-u|$ such that the expansion
    $$f(x,u+v)=\sum_{n=1}^{\infty}\frac{v^n}{(1-u)^{n+1}},\quad \forall u\in D',v\in B_r$$
    holds. Therefore, $f$ is uniformly analytic in $u\in D$ with respect to $x$.

    Let $U=\{u\in L^{\infty}(\Omega;\mathbb{R})\mid u(\Omega)\subset\subset D\ \mathrm{a.e.}\}$. By Theorem \ref{lemma2.3}, the mapping $S(u):=1+\int_{\Omega}\frac{1}{1-u(x)}\,\mathrm{d}x$ is analytic on $U$.

    Note that $\psi\in U$. Take $V$ to be a sufficiently small neighborhood of $\psi$ in $H^2\cap H_0^1(\Omega)$, so that $V\subset U$. Thus $S(u)$ is analytic on $V$.

    Indeed, since $S(u)$ is analytic on $U$, for every $u\in U$ there exist $r>0$ and a sequence of $n$-linear continuous symmetric maps $(M_n)_{n\in \mathbb{N}}$ such that:
    $$\sum_{n=1}^{\infty}\|M_n\|_{\mathcal{L}_n(L^{\infty},\mathbb{R})}r^n<\infty,\ \ \text{where}\ \|M_n\|_{\mathcal{L}_n(L^{\infty},\mathbb{R})}=\sup\{|M_n(u_1,\cdots,u_n)|\ \mid\ \|u_i\|_{L^{\infty}}\le1\},$$
    $$\forall h\in\bar{B}(0,r)\subset U,\ \ \ S(u+h)=S(u)+\sum_{n=1}^{\infty}M_n(h,\cdots,h).$$
    Since $H^2\cap H_0^1(\Omega)$ is compactly embedded in $L^\infty(\Omega)$ for $1\le N\le 3$, we have
    $$\|u\|_{L^{\infty}}\le C\|u\|_{H^2},\quad \forall u\in H^2\cap H_0^1(\Omega).$$
    For every $u\in V\subset U$, since
    $$\|M_n\|_{\mathcal{L}_n(H^2\cap H_0^1,\mathbb{R})}=\sup\{|M_n(u_1,\cdots,u_n)|\ \mid\ \|u_i\|_{H^2}\le1\},$$
    let $v_i=\frac{u_i}{\|u_i\|_{L^{\infty}}}\in H^2\cap H_0^1(\Omega)$ so that $\|v_i\|_{L^{\infty}}=1$. Then
    \begin{align*}
        \|M_n\|_{\mathcal{L}_n(H^2\cap H_0^1,\mathbb{R})}&=\sup\{|M_n(v_1,\cdots,v_n)|\|u_1\|_{L^{\infty}}\cdots\|u_n\|_{L^{\infty}}\ \mid\ \|u_i\|_{H^2}\le1\}\\
        &\le \sup\{\|M_n\|_{\mathcal{L}_n(L^{\infty},\mathbb{R})}C\|u_1\|_{H^2}\cdots C\|u_n\|_{H^2}\ \mid\ \|u_i\|_{H^2}\le1\}\\
        &\le C^n\|M_n\|_{\mathcal{L}_n(L^{\infty},\mathbb{R})}.
    \end{align*}
    Take $r_1<\frac{r}{C}$, then
    $$\sum_{n=1}^{\infty}\|M_n\|_{\mathcal{L}_n(H^2,\mathbb{R})}r_1^n< \sum_{n=1}^{\infty}\|M_n\|_{\mathcal{L}_n(L^{\infty},\mathbb{R})}r^n<\infty,$$
    and thus
    $$\forall h\in\bar{B}(0,r_1)\subset V,\ \ \ S(u+h)=S(u)+\sum_{n=1}^{\infty}M_n(h,\cdots,h).$$
    Therefore, $S(u)$ is analytic on $V$.

    Define the mapping $g:S(V)\to \mathbb{R}$ as:
    $$g(y)=\frac{1}{y}.$$
    For every $y_0\in S(V)$, we have
    $$\frac{1}{y}=\frac{1}{y_0}\cdot \frac{1}{1+\frac{y-y_0}{y_0}}=\frac{1}{y_0}\sum_{n=0}^{\infty}(-1)^n \left(\frac{y-y_0}{y_0}\right)^n,\quad\text{when}\ \left|\frac{y-y_0}{y_0}\right|<1.$$
    Therefore, $g$ is analytic on $S(V)$, and by Theorem \ref{lemma2.2}, $\frac{1}{1+\int_{\Omega}\frac{1}{1-u}\,\mathrm{d}x}$ is analytic on $V$.

    On the other hand, consider $T(u)=\frac{1}{2}\int_{\Omega}|\nabla u|^2\,\mathrm{d}x:V\to\mathbb{R}$. Its derivatives are:
    $$DT(u)(v)=\int_{\Omega}\nabla u\nabla v\,\mathrm{d}x,$$
    $$D^2T(u)(v,w)=\int_{\Omega}\nabla w\nabla v\,\mathrm{d}x.$$
    Thus for every $u\in V$, there exist $r_1>0$ and $n$-linear continuous symmetric maps $DT(u)$ and $D^2T(u)$ such that
    $$\forall h\in\bar{B}(0,r_1)\subset V,\ \ \ T(u+h)=T(u)+DT(u)(h)+D^2T(u)(h,h).$$
    Therefore, $T(u)$ is analytic on $V$.

    In conclusion, the energy functional
    $$E(u)=\frac{1}{2}\int_{\Omega}|\nabla u|^2\,\mathrm{d}x+\frac{\lambda}{1+\int_{\Omega}\frac{1}{1-u}\,\mathrm{d}x}$$
    is analytic on $V$.
\end{proof}

\subsection{Lojasiewicz-Simon Inequality}

\begin{proposition}\label{proposition2.3}
    Let $\psi$ be an equilibrium point satisfying (\ref{1.5}). Then the operator $A=D^2E(\psi):H^2\cap H_0^1(\Omega)\to L^{2}(\Omega)$ is a semi-Fredholm operator.
\end{proposition}

\begin{proof}
    The energy functional is $E(u)=\frac{1}{2}\int_{\Omega}|\nabla u|^2\,\mathrm{d}x+\frac{\lambda}{1+\int_{\Omega}\frac{1}{1-u}\,\mathrm{d}x}$.
    By calculation, we obtain:
    \begin{equation}
        DE(u)(v)=\int_{\Omega}\nabla u\nabla v\,\mathrm{d}x-\int_{\Omega}\frac{\lambda v}{(1+\int_{\Omega}\frac{1}{1-u}\,\mathrm{d}x)^2(1-u)^2}\,\mathrm{d}x,\label{2.10}
    \end{equation}
    \begin{align*}
        D^2E(u)(v,w)=&\int_{\Omega}\nabla w\nabla v\,\mathrm{d}x+2\int_{\Omega}\frac{\lambda v}{(1+\int_{\Omega}\frac{1}{1-u}\,\mathrm{d}x)^3(1-u)^2}\left(\int_{\Omega}\frac{w}{(1-u)^2}\,\mathrm{d}x\right)\,\mathrm{d}x\\
        &-2\int_{\Omega}\frac{\lambda v}{(1+\int_{\Omega}\frac{1}{1-u}\,\mathrm{d}x)^2}\frac{w}{(1-u)^3}\,\mathrm{d}x.
    \end{align*}

    For the equilibrium point $\psi$, the operator $A:H^2\cap H_0^1(\Omega)\to L^{2}(\Omega)$ satisfies
    $$\langle Av,w \rangle = D^2E(\psi)(v,w),$$
    and in the weak sense,
    $$Av=\frac{2\int_{\Omega}\frac{v}{(1-\psi)^2}\,\mathrm{d}x\cdot\lambda}{(1+\int_{\Omega}\frac{1}{1-\psi}\,\mathrm{d}x)^3(1-\psi)^2}-\frac{2\lambda v}{(1+\int_{\Omega}\frac{1}{1-\psi}\,\mathrm{d}x)^2(1-\psi)^3}-\Delta v.$$

    Let $$G(v)=\frac{2\int_{\Omega}\frac{v}{(1-\psi)^2}\,\mathrm{d}x\cdot\lambda}{(1+\int_{\Omega}\frac{1}{1-\psi}\,\mathrm{d}x)^3(1-\psi)^2}-\frac{2\lambda v}{(1+\int_{\Omega}\frac{1}{1-\psi}\,\mathrm{d}x)^2(1-\psi)^3}.$$
    We now prove that $G:H^2\cap H_0^1(\Omega)\to H_0^1(\Omega)$ is a bounded linear operator:

    For every $v\in H^2\cap H_0^1(\Omega)$,
    \begin{align*}
        \int_{\Omega}|G(v)(x)|^2\,\mathrm{d}x=&\int_{\Omega}\frac{4\lambda^2(\int_{\Omega}\frac{v}{(1-\psi)^2}\,\mathrm{d}x)^2}{(1+\int_{\Omega}\frac{1}{1-\psi}\,\mathrm{d}x)^6(1-\psi)^4}-\frac{8\lambda^2v\int_{\Omega}\frac{v}{(1-\psi)^2}\,\mathrm{d}x}{(1+\int_{\Omega}\frac{1}{1-\psi}\,\mathrm{d}x)^5(1-\psi)^5}\\
        &+\frac{4\lambda^2v^2}{(1+\int_{\Omega}\frac{1}{1-\psi}\,\mathrm{d}x)^4(1-\psi)^6}\,\mathrm{d}x.
    \end{align*}
    Since $\psi$ satisfies (\ref{1.5}), $\|\frac{1}{1-\psi}\|_{L^\infty}$ is bounded. Moreover, for $1\le N\le 3$, $H^2\cap H_0^1(\Omega)$ is compactly embedded in $L^\infty(\Omega)$, so
    $$\int_{\Omega}\frac{v}{(1-\psi)^2}\,\mathrm{d}x\le M\|v\|_{L^\infty}\le M\|v\|_{H^2}.$$
    Thus
    $$\int_{\Omega}|G(v)(x)|^2\,\mathrm{d}x\le \int_{\Omega}\frac{M\|v\|_{H^2}^2}{(1-\psi)^4}+\frac{M\|v\|_{H^2}|v|}{(1-\psi)^5}+\frac{M v^2}{(1-\psi)^6}\,\mathrm{d}x\le C\|v\|_{H^2}^2.$$
    Furthermore,
    $$\nabla G(v)=\frac{4\lambda\int_{\Omega}\frac{v}{(1-\psi)^2}\,\mathrm{d}x}{(1+\int_{\Omega}\frac{1}{1-\psi}\,\mathrm{d}x)^3}\frac{\nabla \psi}{(1-\psi)^3}-\frac{2\lambda}{(1+\int_{\Omega}\frac{1}{1-\psi}\,\mathrm{d}x)^2}\frac{(1-\psi)\nabla v+3v\nabla\psi}{(1-\psi)^4},$$
    so
    \begin{align*}
        &\quad\int_{\Omega}|\nabla G(v)(x)|^2\,\mathrm{d}x\\
        &=\int_{\Omega}\frac{16\lambda^2(\int_{\Omega}\frac{v}{(1-\psi)^2}\,\mathrm{d}x)^2}{(1+\int_{\Omega}\frac{1}{1-\psi}\,\mathrm{d}x)^6}\frac{|\nabla \psi|^2}{(1-\psi)^6}-\frac{16\lambda^2\int_{\Omega}\frac{v}{(1-\psi)^2}\,\mathrm{d}x}{(1+\int_{\Omega}\frac{1}{1-\psi}\,\mathrm{d}x)^5}\frac{(1-\psi)\nabla\psi\nabla v+3v|\nabla \psi|^2}{(1-\psi)^7}\\
        &\quad+\frac{4\lambda^2}{(1+\int_{\Omega}\frac{1}{1-\psi}\,\mathrm{d}x)^4}\frac{|\nabla v|^2(1-\psi)^2+6v(1-\psi)\nabla\psi\nabla v+9v^2|\nabla \psi|^2}{(1-\psi)^8}\,\mathrm{d}x\\
        &\le\int_{\Omega}M\|v\|_{H^2}^2\frac{|\nabla \psi|^2}{(1-\psi)^6}+M\|v\|_{H^2}\frac{(1-\psi)|\nabla\psi\nabla v|+3|v||\nabla \psi|^2}{(1-\psi)^7}\\
        &\quad+M\frac{|\nabla v|^2(1-\psi)^2+6|v|(1-\psi)|\nabla\psi\nabla v|+9v^2|\nabla \psi|^2}{(1-\psi)^8}\,\mathrm{d}x\\
        &\le M\|v\|_{H^2}^2\int_{\Omega}\frac{|\nabla \psi|^2}{(1-\psi)^6}\,\mathrm{d}x+M\|v\|_{H^2}\left(\int_{\Omega}\frac{|\nabla\psi\nabla v|}{(1-\psi)^6}\,\mathrm{d}x+\int_{\Omega}\frac{3|v||\nabla \psi|^2}{(1-\psi)^7}\,\mathrm{d}x\right)\\
        &\quad+M\left(\int_{\Omega}\frac{|\nabla v|^2}{(1-\psi)^6}\,\mathrm{d}x+\int_{\Omega}\frac{6|v\nabla\psi\nabla v|}{(1-\psi)^7}\,\mathrm{d}x+\int_{\Omega}\frac{9v^2}{(1-\psi)^8}\,\mathrm{d}x\right)\\
        &\le M\|v\|_{H^2}^2\|\nabla \psi\|_{L^2}^2+M\|v\|_{H^2}\left(\|\nabla \psi\|_{L^2}\|\nabla v\|_{L^2}+\|v\|_{L^\infty}\|\nabla\psi\|_{L^2}^2\right)\\
        &\quad+M\left(\|\nabla v\|_{L^2}^2+\|v\|_{L^\infty}\|\nabla \psi\|_{L^2}\|\nabla v\|_{L^2}+\|v\|_{L^2}^2\right)\\
        &\le C\|v\|_{H^2}^2.
    \end{align*}
    Therefore
    $$\|G(v)\|_{H^1}^2=\int_{\Omega}|G(v)(x)|^2\,\mathrm{d}x+\int_{\Omega}|\nabla G(v)(x)|^2\,\mathrm{d}x\le C\|v\|_{H^2}^2.$$
    On the other hand, the embedding operator $I:H_0^1(\Omega)\to L^{2}(\Omega)$, $I(v)=v$ is compact. Hence $G:H^2\cap H_0^1(\Omega)\to L^{2}(\Omega)$ is a compact operator.

    By Remark \ref{delta}, $-\Delta:H^2\cap H_0^1(\Omega)\to L^{2}(\Omega)$ is a semi-Fredholm operator. Combined with Theorem \ref{lemma2.4}, $A=G-\Delta$ is a semi-Fredholm operator.
\end{proof}

\begin{theorem}\label{theorem2.1}
    Let $\mathscr{E}$ be the set of equilibrium points of equation (\ref{1.1}):
    $$\mathscr{E}=\left\{\phi\mid -\Delta\phi=\frac{\lambda}{(1-\phi)^2(1+\int_{\Omega}\frac{1}{1-\phi}\,\mathrm{d}x)^2},\  \phi|_{\partial\Omega}=0\right\},$$
    and let $\psi\in\mathscr{E}$. Then there exist $\sigma>0$ depending on $\psi$ and $\theta\in (0,\frac{1}{2})$ such that for every $u\in H^2\cap H_0^1(\Omega)$ satisfying $\|u-\psi\|_{H^2}\le\sigma$:
    \begin{equation}
        \left\|\Delta u+\frac{\lambda}{(1-u)^2(1+\int_{\Omega}\frac{1}{1-u}\,\mathrm{d}x)^2}\right\|_{L^2}\ge |E(u)-E(\psi)|^{1-\theta}.\label{2.11}
    \end{equation}
\end{theorem}

\begin{proof}
    Consider the linearization near the equilibrium point $\psi$:
    \begin{equation}
        \begin{cases}
            Lw\equiv -\Delta w+\frac{2\int_{\Omega}\frac{w}{(1-\psi)^2}\,\mathrm{d}x\cdot\lambda}{(1+\int_{\Omega}\frac{1}{1-\psi}\,\mathrm{d}x)^3(1-\psi)^2}-\frac{2\lambda w}{(1+\int_{\Omega}\frac{1}{1-\psi}\,\mathrm{d}x)^2(1-\psi)^3}=0,\ x\in\Omega \\
            w |_{\partial\Omega}=0. \label{2.12}
        \end{cases}
    \end{equation}
    By Proposition \ref{proposition2.3}, $L:H^2\cap H_0^1(\Omega)\to L^{2}(\Omega)$ is a semi-Fredholm operator, so its kernel $\ker(L)$ is finite-dimensional.

    Let $\dim\ker (L)=m$, and let $(\phi_1,\cdots,\phi_m)$ be an orthonormal basis of $\ker(L)$. Let $\Pi:H^2\cap H_0^1(\Omega) \to \ker(L)\subset L^{2}(\Omega)$ be the projection operator. Define the operator $\tilde{L}:H^2\cap H_0^1(\Omega)\to L^{2}(\Omega)$ by
    $$\tilde{L}w=\Pi w+Lw,$$
    which is bijective.

    Decompose $u$ as:
    $$u=v+\psi,\quad v\in H^2\cap H_0^1(\Omega).$$
    Define the operator $\mathcal{M}:H^2\cap H_0^1(\Omega)\to L^{2}(\Omega)$:
    \begin{equation}
        \mathcal{M}(v)=-\Delta u-\frac{\lambda}{(1-u)^2(1+\int_{\Omega}\frac{1}{1-u}\,\mathrm{d}x)^2},\quad \text{where}\ u=v+\psi,\label{2.13}
    \end{equation}
    so that $D\mathcal{M}(0)=L$.

    Let
    \begin{equation}
        \mathcal{N}(v)=\mathcal{M}(v)+\Pi v,\label{2.14}
    \end{equation}
    then $D\mathcal{N}(0)=\tilde{L}$.

    Since $\tilde{L}$ is bijective, by the local inverse mapping theorem, there exist a neighborhood $W_1(0)\subset H^2\cap H_0^1(\Omega)$ of $0$ and a neighborhood $W_2(0)\subset L^{2}(\Omega)$ of $0$, and an inverse mapping
    $$\Psi:W_2(0)\to W_1(0)$$
    such that
    \begin{equation}
        \mathcal{N}(\Psi(g))=g,\quad \forall g\in W_2(0),\label{2.15}
    \end{equation}
    \begin{equation}
        \Psi(\mathcal{N}(v))=v,\quad \forall v\in W_1(0),\label{2.16}
    \end{equation}
    with
    \begin{equation}
        \|\mathcal{N}(v_1)-\mathcal{N}(v_2)\|_{L^{2}}\le C\|v_1-v_2\|_{H^2},\quad \forall v_1,v_2\in W_1(0),\label{2.17}
    \end{equation}
    \begin{equation}
        \|\Psi(g_1)-\Psi(g_2)\|_{H^2}\le C\|g_1-g_2\|_{L^{2}},\quad \forall g_1,g_2\in W_2(0),\label{2.18}
    \end{equation}
    where $C>0$ is a constant.

    Since $\frac{1}{1-u}$ and $\frac{1}{1+\int_{\Omega}\frac{1}{1-u}\,\mathrm{d}x}$ are analytic in a neighborhood of $\psi\in H^2\cap H_0^1(\Omega)$, $\mathcal{M}$ is analytic in a neighborhood of $0\in H^2\cap H_0^1(\Omega)$. Thus $\mathcal{N}:W_1(0)\to W_2(0)$ and $\Psi:W_2(0)\to W_1(0)$ are both analytic.

    Let
    $$\xi=(\xi_1,\cdots,\xi_m)\in \mathbb{R}^m,\quad \Pi v=\sum_{j=1}^{m}\xi_j \phi_j.$$
    Define $\Gamma:\mathbb{R}^m\to \mathbb{R}$ by
    $$\Gamma(\xi)=E\left(\Psi\left(\sum_{j=1}^{m}\xi_j \phi_j\right)+\psi\right),\quad \text{for sufficiently small}\ \xi\ \text{such that}\ \sum_{j=1}^{m}\xi_j \phi_j\in W_2(0).$$
    By the analyticity of $E$ and $\Psi$, $\Gamma(\xi)$ is analytic in a neighborhood of $0\in \mathbb{R}^m$.

    Differentiating (\ref{2.15}) with respect to $g$ gives:
    $$D\mathcal{N}(v)\cdot D\Psi(g)=I,$$
    i.e., $D\mathcal{N}(v)\in \mathcal{L}(H^2\cap H_0^1,L^{2})$ and $D\Psi(g)\in \mathcal{L}(L^{2},H^2\cap H_0^1)$.

    Partial derivatives of $\Gamma$:
    \begin{equation}
        \frac{\partial\Gamma}{\partial\xi_j}=DE\left(\Psi\left(\sum_{j=1}^{m}\xi_j \phi_j\right)+\psi\right)\cdot D\Psi \left(\sum_{j=1}^{m}\xi_j \phi_j\right)\cdot \phi_j,\quad j=1,\cdots,m.\label{partial}
    \end{equation}
    where
    \begin{align*}
        DE(u)\cdot v&=\int_{\Omega}\nabla u\nabla v\,\mathrm{d}x-\int_{\Omega}\frac{\lambda v}{(1+\int_{\Omega}\frac{1}{1-u}\,\mathrm{d}x)^2(1-u)^2}\,\mathrm{d}x\\
        &=\int_{\Omega}\left(-\Delta u-\frac{\lambda}{(1+\int_{\Omega}\frac{1}{1-u}\,\mathrm{d}x)^2(1-u)^2}\right)v\,\mathrm{d}x.
    \end{align*}
    Since $\psi\in\mathscr{E}$, $\xi=0$ is a critical point of $\Gamma(\xi)$.

    By the Lojasiewicz Inequality, there exist small constants $\sigma>0$ and $\theta\in (0,\frac{1}{2})$ such that
    $$|\nabla\Gamma(\xi)|\ge |\Gamma(\xi)-\Gamma(0)|^{1-\theta},$$
    i.e.,
    \begin{equation}
        |\nabla\Gamma(\xi)|\ge |\Gamma(\xi)-E(\psi)|^{1-\theta}.\label{2.19}
    \end{equation}

    From (\ref{2.13}) and (\ref{partial}),
    \begin{align*}
        \left|\frac{\partial\Gamma}{\partial\xi_j}\right|&\le \left\|\mathcal{M}\left(\Psi\left(\sum_{j=1}^{m}\xi_j \phi_j\right)\right)\right\|_{L^{2}}\left\|D\Psi \left(\sum_{j=1}^{m}\xi_j \phi_j\right)\right\|_{\mathcal{L}(L^{2},H^2\cap H_0^1)}\|\phi_j\|_{L^{2}}\\
        &\le C\left\|\mathcal{M}\left(\Psi\left(\sum_{j=1}^{m}\xi_j \phi_j\right)\right)\right\|_{L^{2}},\quad j=1,\dots,m.
    \end{align*}
    Thus
    \begin{equation}
        \begin{aligned}
            |\nabla\Gamma(\xi)|&\le C\left\|\mathcal{M}\left(\Psi\left(\Pi v\right)\right)\right\|_{L^{2}}\\
            &=C\left\|\mathcal{M}\left(\Psi\left(\Pi v\right)\right)-\mathcal{M}(v)+\mathcal{M}(v)\right\|_{L^{2}}\\
            &\le C\left(\left\|\mathcal{M}\left(\Psi\left(\Pi v\right)\right)-\mathcal{M}(v)\right\|_{L^{2}}+\|\mathcal{M}(v)\|_{L^{2}}\right),\quad v\in W_1(0).
        \end{aligned}\label{2.20}
    \end{equation}
    From (\ref{2.14}) and (\ref{2.17}),
    \begin{equation}
        \begin{aligned}
            \|\mathcal{M}(v_1)-\mathcal{M}(v_2)\|_{L^{2}}&\le \|\Pi v_1-\Pi v_2\|_{L^{2}}+\|\mathcal{N}(v_1)-\mathcal{N}(v_2)\|_{L^{2}}\\
            &\le \|v_1-v_2\|_{L^{2}}+C\|v_1-v_2\|_{H^2}\\
            &\le (C+1)\|v_1-v_2\|_{H^2},\quad v_1,v_2\in W_1(0).
        \end{aligned}\label{2.21}
    \end{equation}
    Similarly,
    \begin{equation}
        \begin{aligned}
            \|\Psi(\Pi v)-v\|_{H^2}&=\|\Psi(\Pi v)-\Psi(\mathcal{N}v)\|_{H^2}\\
            &\le C\|\Pi v-\mathcal{N}v\|_{L^{2}}\\
            &=C\|\mathcal{M}(v)\|_{L^{2}}.
        \end{aligned}\label{2.22}
    \end{equation}
    Combining (\ref{2.20})-(\ref{2.22}):
    \begin{equation}
        \begin{aligned}
            |\nabla\Gamma(\xi)|&\le C\left((C+1)\|\Psi(\Pi v)-v\|_{H^2}+\|\mathcal{M}(v)\|_{L^{2}}\right)\\
            &\le C\left((C+1)C\|\mathcal{M}(v)\|_{L^{2}}+\|\mathcal{M}(v)\|_{L^{2}}\right)\\
            &\le C\|\mathcal{M}(v)\|_{L^{2}}.
        \end{aligned}\label{2.23}
    \end{equation}

    On the other hand, for $t\in [0,1]$ and $v\in W_1(0)$, we have $v+t(\Psi(\Pi v)-v)\in W_1(0)$. Thus for $u=v+\psi$, using (\ref{2.21}) and (\ref{2.22}):
    \begin{align*}
        |E(u)-\Gamma(\xi)|&=|E(u)-E(\Psi(\Pi v)+\psi)|\\
            &=\left|\int_{0}^{1}\frac{\mathrm{d}}{\mathrm{d}t}E(u+(1-t)(\Psi(\Pi v)-v))\,\mathrm{d}t\right|\\
            &=\left|\int_{0}^{1}DE(u+(1-t)(\Psi(\Pi v)-v))\cdot (\Psi(\Pi v)-v)\,\mathrm{d}t\right|\\
            &\le \max_{0\le t\le1}\left|DE(u+(1-t)(\Psi(\Pi v)-v))\cdot (\Psi(\Pi v)-v)\right|.
    \end{align*}
    Let $z=u+(1-t)(\Psi(\Pi v)-v)$. By (\ref{2.10}):
    \begin{align*}
        &\quad DE(u+(1-t)(\Psi(\Pi v)-v))\cdot (\Psi(\Pi v)-v)\\
        &=DE(z)\cdot (\Psi(\Pi v)-v)\\
            &=\int_{\Omega}\nabla z\nabla (\Psi(\Pi v)-v)\,\mathrm{d}x-\int_{\Omega}\frac{\lambda (\Psi(\Pi v)-v)}{(1+\int_{\Omega}\frac{1}{1-z}\,\mathrm{d}x)^2(1-z)^2}\,\mathrm{d}x\\
            &=-\int_{\Omega}\Delta z\cdot (\Psi(\Pi v)-v)-\frac{\lambda (\Psi(\Pi v)-v)}{(1+\int_{\Omega}\frac{1}{1-z}\,\mathrm{d}x)^2(1-z)^2}\,\mathrm{d}x\\
            &=-\int_{\Omega}(\Psi(\Pi v)-v)\left(\Delta z-\frac{\lambda}{(1+\int_{\Omega}\frac{1}{1-z}\,\mathrm{d}x)^2(1-z)^2}\right)\,\mathrm{d}x\\
            &=-\int_{\Omega}(\Psi(\Pi v)-v)(\mathcal{M}(z-\psi))\,\mathrm{d}x.
    \end{align*}
    Therefore
    \begin{equation}
        \begin{aligned}
            |E(u)-\Gamma(\xi)|&\le \max_{0\le t\le1}\|\mathcal{M}(v+(1-t)(\Psi(\Pi v)-v))\|_{L^{2}}\|\Psi(\Pi v)-v\|_{H^{2}}\\
            &= \max_{0\le t\le1}\left\|\mathcal{M}(v+(1-t)(\Psi(\Pi v)-v))-\mathcal{M}(v)+\mathcal{M}(v)\right\|_{L^{2}}\|\Psi(\Pi v)-v\|_{H^{2}}\\
            &\le \max_{0\le t\le1}\left((1-t)(C+1)\|(\Psi(\Pi v)-v)\|_{H^{2}}+\|\mathcal{M}(v)\|_{L^{2}}\right)\|\Psi(\Pi v)-v\|_{H^{2}}\\
            &\le \left(\|\mathcal{M}(v)\|_{L^{2}}+(C+1)C\|\mathcal{M}(v)\|_{L^{2}}\right)\cdot C\|\mathcal{M}(v)\|_{L^{2}}\\
            &\le C\|\mathcal{M}(v)\|_{L^{2}}^2.
        \end{aligned}\label{2.24}
    \end{equation}

    Substituting (\ref{2.23})-(\ref{2.24}) into (\ref{2.19}) and using $|a+b|^{1-\theta}\le 2|a|^{1-\theta}+|b|^{1-\theta}$:

    For $v\in W_1(0)$,
    \begin{align*}
        C\|\mathcal{M}(v)\|_{L^{2}}&\ge |\nabla\Gamma(\xi)|\\
        &\ge |\Gamma(\xi)-E(\psi)|^{1-\theta}\\
        &\ge \frac{1}{2}\left(|E(u)-E(\psi)|^{1-\theta}-|E(u)-\Gamma(\xi)|^{1-\theta}\right)\\
        &\ge \frac{1}{2}|E(u)-E(\psi)|^{1-\theta}-\frac{1}{2}C^{1-\theta}\|\mathcal{M}(v)\|_{L^{2}}^{2(1-\theta)}.
    \end{align*}
    Thus
    $$C\|\mathcal{M}(v)\|_{L^{2}}+\frac{1}{2}C^{1-\theta}\|\mathcal{M}(v)\|_{L^{2}}^{2(1-\theta)}\ge \frac{1}{2}|E(u)-E(\psi)|^{1-\theta}.$$

    Since $\theta\in (0,\frac{1}{2})$, $2(1-\theta)> 1$, and $\|\mathcal{M}(v)\|_{L^{2}}$ is bounded, we obtain:
    $$\|\mathcal{M}(v)\|_{L^{2}}\ge C|E(u)-E(\psi)|^{1-\theta},\quad v\in W_1(0).$$

    Let $\epsilon>0$ be a small constant. We can choose $\sigma$ sufficiently small so that
    $$C|E(u)-E(\psi)|^{-\epsilon}\ge 1,\quad \|u-\psi\|_{H^2}\le \sigma.$$
    Therefore
    $$\|\mathcal{M}(v)\|_{L^{2}}\ge |E(u)-E(\psi)|^{1-\theta+\epsilon}=|E(u)-E(\psi)|^{1-\theta'},$$
    where $\|v\|_{H^2}=\|u-\psi\|_{H^2}\le \sigma$.\\
    Hence
    $$\left\|\Delta u+\frac{\lambda}{(1-u)^2(1+\int_{\Omega}\frac{1}{1-u}\,\mathrm{d}x)^2}\right\|_{L^{2}}\ge |E(u)-E(\psi)|^{1-\theta},$$
    for $\|u-\psi\|_{H^2}\le \sigma,\ \theta\in (0,\frac{1}{2})$. The theorem is proved.
\end{proof}

\begin{remark}\label{remark2.1}
    If we modify (\ref{2.11}) to
    \begin{equation}
        \left\|\Delta u+\frac{\lambda}{(1-u)^2(1+\int_{\Omega}\frac{1}{1-u}\,\mathrm{d}x)^2}\right\|_{L^{2}}\ge C|E(u)-E(\psi)|^{1-\theta}\label{2.25}
    \end{equation}
    where $C>0$ is a constant, then $\theta$ can be taken as $\frac{1}{2}$.
\end{remark}

\subsection{Asymptotic behavior}

To prove Theorem \ref{theorem1.1}, we first introduce the following lemma:

\begin{lemma}[See Theorem 6.3.2 in \cite{ref8}]\label{lemma2.5}
    Let $(H,S(t),V)$ be a gradient system. Then for every $x\in H$, the $\omega$-limit set $\omega(x)$ is a connected compact invariant set consisting entirely of fixed points of $S(t)$.
\end{lemma}

\begin{proof}[Proof of Theorem \ref{theorem1.1}]
    In Section 5.1, we constructed the gradient system $(H^2\cap H_0^1(\Omega),S(t),E)$. By Lemma \ref{lemma2.5}, for every $\lambda>0$ and $u_0\in H^2\cap H_0^1(\Omega)$, the $\omega$-limit set $\omega(u_0)$ consists entirely of equilibrium points, i.e., there exists $\psi\in \omega(u_0)$ satisfying equation (\ref{1.2}), and there exists a sequence $t_n\to \infty$ such that $\lim_{n \to +\infty} \|u(\cdot,t_n)-\psi\|_{H^2}=0$.

    We can show that $\psi\in C^{\infty}(\Omega)$. Indeed, since for $1\le N\le 3$, $H^2\cap H_0^1(\Omega)\hookrightarrow C^{0,\alpha}(\Omega)$ for $0<\alpha<1$, we have $\psi\in C^{0,\alpha}(\Omega)$. Thus there exists $\delta\in (0,1)$ such that $\|\psi\|_{L^{\infty}}\le 1-\delta$. Therefore, for every $x,\ y\in \Omega$,
    $$\left|\frac{1}{1-\psi(x)}-\frac{1}{1-\psi(y)}\right|=\left|\frac{\psi(x)-\psi(y)}{(1-\psi(x))(1-\psi(y))}\right|\le C|\psi(x)-\psi(y)|\le C|x-y|^{\alpha},$$
    so $\frac{1}{1-\psi}\in C^{0,\alpha}(\Omega)$, and hence $f(\psi)=\frac{\lambda}{(1-\psi)^2(1+\int_{\Omega}\frac{1}{1-\psi}\,\mathrm{d}x)^2}\in C^{0,\alpha}(\Omega)$. By Schauder estimates, $\psi\in C^{2,\alpha}(\Omega)$. Iterating this process, we obtain $\psi\in C^{\infty}(\Omega)$.

    Next, we prove that $\omega(u_0)$ consists of a unique equilibrium point.

    From (\ref{2.2}), $E(u(t))$ is monotonically decreasing in $t$ and $E(u(t))\ge 0$, so
    $$\lim_{t\to\infty}E(u(t))=E_{\infty}.$$
    Moreover, for every $u,\ v\in H^2\cap H_0^1(\Omega)$ with $\|u\|_{L^{\infty}},\ \|v\|_{L^{\infty}}\le 1-\delta$,
    \begin{align*}
        |E(u)-E(v)|&=\left|\frac{1}{2}\int_{\Omega}|\nabla u|^2-|\nabla v|^2\,\mathrm{d}x+\frac{\lambda}{1+\int_{\Omega}\frac{1}{1-u}\,\mathrm{d}x}-\frac{\lambda}{1+\int_{\Omega}\frac{1}{1-v}\,\mathrm{d}x}\right|\\
        &= \left|\frac{1}{2}\int_{\Omega}|\nabla (u+v)||\nabla (u-v)|\,\mathrm{d}x+\frac{\lambda\int_{\Omega}\frac{v-u}{(1-u)(1-v)}\,\mathrm{d}x}{(1+\int_{\Omega}\frac{1}{1-u}\,\mathrm{d}x)(1+\int_{\Omega}\frac{1}{1-v}\,\mathrm{d}x)}\right|\\
        &\le \|u+v\|_{H^2}\|u-v\|_{H^2}+C\|u-v\|_{H^2},
    \end{align*}
    hence
    $$\lim_{n\to\infty}|E(u(t_n))-E(\psi)|\le \lim_{n\to\infty}(\|u(t_n)+\psi\|_{H^2}\|u(t_n)-\psi\|_{H^2}+C\|u(t_n)-\psi\|_{H^2})=0.$$
    Therefore, $$E(\psi)=E_{\infty}.$$
    Similar to (\ref{2.3}), we also have:
    $$E(\psi)+\int_{t}^{\infty}\|u_t\|_{L^2}^{2}\,\mathrm{d}\tau=E(u(t)).$$
    If there exists $t_0>0$ such that $E(u(t_0))=E(\psi)$, then $\forall\ t\ge t_0$, $u_t=0$ and $u=\psi$, proving (\ref{1.6}).

    Now consider the case where for all $t\ge 0$, $E(u(t))>E(\psi)$.

    By Theorem \ref{theorem2.1}, for $\psi\in\omega(u_0)$, there exist $\sigma>0$ and $\theta\in (0,\frac{1}{2})$ depending on $\psi$ such that when $\|u-\psi\|_{H^2}\le\sigma$:
    \begin{equation}
        \left\|\Delta u+\frac{\lambda}{(1-u)^2(1+\int_{\Omega}\frac{1}{1-u}\,\mathrm{d}x)^2}\right\|_{L^{2}}\ge |E(u)-E(\psi)|^{1-\theta}.\label{2.26}
    \end{equation}
    Note that (\ref{2.2}) is equivalent to
    \begin{equation}
        \frac{\mathrm{d}}{\mathrm{d}t}(E(u(t))-E(\psi))+\|u_t\|_{L^2}\left\|-\Delta u-\frac{\lambda}{(1-u)^2(1+\int_{\Omega}\frac{1}{1-u}\,\mathrm{d}x)^2}\right\|_{L^2}=0,\label{2.27}
    \end{equation}
    and
    \begin{equation}
        \frac{\mathrm{d}}{\mathrm{d}t}(E(u(t))-E(\psi))^{\theta}=\theta(E(u(t))-E(\psi))^{\theta-1} \frac{\mathrm{d}}{\mathrm{d}t}(E(u(t))-E(\psi)).\label{2.28}
    \end{equation}
    Substituting (\ref{2.27}) into (\ref{2.28}) yields:
    \begin{equation}
        \begin{aligned}
            &\quad\ \ \frac{\mathrm{d}}{\mathrm{d}t}(E(u(t))-E(\psi))^{\theta}\\
            &=\ -\theta(E(u(t))-E(\psi))^{\theta-1}\|u_t\|_{L^2}\left\|-\Delta u-\frac{\lambda}{(1-u)^2(1+\int_{\Omega}\frac{1}{1-u}\,\mathrm{d}x)^2}\right\|_{L^2}.
        \end{aligned}\label{2.29}
    \end{equation}
    If there exists a sufficiently large $T>0$ such that $\forall\ t>T$, $\|u-\psi\|_{H^2}<\sigma$, then by (\ref{2.26}):
    \begin{align*}
        &\quad\ \ -\frac{1}{\theta}\frac{\mathrm{d}}{\mathrm{d}t}(E(u(t))-E(\psi))^{\theta}\\
        &=\ (E(u(t))-E(\psi))^{\theta-1}\|u_t\|_{L^2}\left\|-\Delta u-\frac{\lambda}{(1-u)^2(1+\int_{\Omega}\frac{1}{1-u}\,\mathrm{d}x)^2}\right\|_{L^2}\\
        &\ge\ \|u_t\|_{L^2}.
    \end{align*}
    Therefore,
    $$\int_{T}^{+\infty}\|u_t\|_{L^2}\,\mathrm{d}\tau <+\infty.$$
    Moreover, since
    $$\|u(t)-u(s)\|_{L^2}=\left\|\int_{s}^{t}u_t\,\mathrm{d}\tau\right\|_{L^2}\le \int_{s}^{t}\|u_t\|_{L^2}\,\mathrm{d}\tau,$$
    we have
    $$\|u(t)-\psi\|_{L^2}\to 0,\quad t\to +\infty.$$
    By the relative compactness of $\bigcup_{t\ge T}u(t)$ in $H^2\cap H_0^1(\Omega)$, (\ref{1.6}) is proved.

    Now we prove the existence of $T$.

    Since in $H^2\cap H_0^1(\Omega)$, $u(\cdot,t_n)\to \psi$ and $E(u(t_n))\to E(\psi)$, for every $\epsilon>0$ with $\epsilon<\sigma$, there exists $N\in \mathbb{N}^{*}$ such that when $n\ge N$:
    \begin{equation}
        \|u(\cdot,t_n)-\psi\|_{H^2}<\frac{\epsilon}{2},\quad \frac{1}{\theta}(E(u(t_n))-E(\psi))^{\theta}<\frac{\epsilon}{2}.\label{2.30}
    \end{equation}
    For $n\ge N$, let
    $$\overline{t_n}=\sup\{t\ge t_n\ \mid\ \|u(\cdot,s)-\psi\|_{H^2}<\sigma,\quad \forall s\in[t_n,t]\}.$$
    If there exists $n_0\ge N$ such that $\overline{t_{n_0}}=+\infty$, then take $T=t_{n_0}$, proving the existence of $T$.\\
    If $\forall\ n\ge N$, $\overline{t_n}<+\infty$,
    then by Theorem \ref{theorem2.1}, for $t\in [t_n,\overline{t_n}]$,
    \begin{equation}
        \begin{aligned}
            &\quad\ \ -\frac{\mathrm{d}}{\mathrm{d}t}(E(u(t))-E(\psi))^{\theta}\\
            &=\ -\theta\frac{\,\mathrm{d}E(u(t))}{\,\mathrm{d}t}(E(u)-E(\psi))^{\theta-1}\\
            &=\ \theta \|u_t\|_{L^2}^2(E(u)-E(\psi))^{\theta-1}\\
            &=\ \theta \|u_t\|_{L^2}\left\|-\Delta u-\frac{\lambda}{(1-u)^2(1+\int_{\Omega}\frac{1}{1-u}\,\mathrm{d}x)^2}\right\|_{L^{2}}(E(u)-E(\psi))^{\theta-1}\\
            &\ge\ \theta\|u_t\|_{L^2}.
        \end{aligned}\label{2.31}
    \end{equation}
    Integrating (\ref{2.31}) with respect to $t$ gives:
    \begin{equation}
        \int_{t_n}^{\overline{t_n}}\|u_t\|_{L^2}\,\mathrm{d}\tau\le \frac{1}{\theta}(E(u(t_n))-E(\psi))^{\theta}.\label{2.32}
    \end{equation}
    Combining (\ref{2.30})-(\ref{2.32}), $\forall\ n\ge N$:
    \begin{align*}
        \|u(\overline{t_n})-\psi\|_{L^2}&\le \|u(\overline{t_n})-u(t_n)\|_{L^2}+\|u(t_n)-\psi\|_{L^2}\\
        &\le \int_{t_n}^{\overline{t_n}}\|u_t\|_{L^2}\,\mathrm{d}\tau+\|u(t_n)-\psi\|_{L^2}\\
        &\le \frac{1}{\theta}(E(u(t_n))-E(\psi))^{\theta}+\|u(t_n)-\psi\|_{H^2}\\
        &\le \epsilon.
    \end{align*}
    That is, as $n\to +\infty$, $u(\overline{t_n})\to \psi$ in $L^2$.

    By the relative compactness of $\bigcup_{t\ge t_0}u(t)$ in $H^2\cap H_0^1(\Omega)$, there exists a subsequence $u(\overline{t_n})$ (still denoted as $u(\overline{t_n})$) such that $u(\overline{t_n})\to \psi$ in $H^2\cap H_0^1(\Omega)$.

    Therefore, there exists $N'\ge N$ such that when $n\ge N'$, $\|u(\overline{t_n})-\psi\|_{H^2}<\frac{\sigma}{2}$, which contradicts the definition of $\overline{t_n}$.

    Theorem \ref{theorem1.1} is proved.
\end{proof}

\subsection{Convergence Rate of Solutions}

\begin{proof}[Proof of Theorem \ref{theorem1.2}]
First, we prove that the solution $u$ of equation (\ref{1.1}) satisfies in $L^2(\Omega)$:
\begin{itemize}
    \item When $\theta=\frac{1}{2}$: $\|u(\cdot,t)-\psi\|_{L^2}\le C_0e^{-C_1t}$ as $t\to +\infty$.
    \item When $\theta\in (0,\frac{1}{2})$: $\|u(\cdot,t)-\psi\|_{L^2}\le C(1+t)^{-\frac{\theta}{1-2\theta}}$ as $t\to +\infty$.
\end{itemize}
  Indeed, by Theorem \ref{theorem2.1} and Remark \ref{remark2.1}, when $t$ is sufficiently large:
  \begin{equation}
    \|u_t\|_{L^2}= \left\|\Delta u+\frac{\lambda}{(1-u)^2(1+\int_{\Omega}\frac{1}{1-u}\,\mathrm{d}x)^2}\right\|_{L^{2}}\ge C|E(u)-E(\psi)|^{1-\theta},\quad \theta\in\left (0,\frac{1}{2}\right ].\label{2.35}
  \end{equation}
  Let $y(t)=(E(u)-E(\psi))^{\theta}$. From (\ref{2.29}) and (\ref{2.35}):
  \begin{equation}
      \frac{\,\mathrm{d}y}{\,\mathrm{d}t}+C\|u_t\|_{L^2}\le 0.\label{2.36}
  \end{equation}
  Integrating (\ref{2.36}):
  $$\int_{t}^{+\infty}\|u_t\|_{L^2}\,\mathrm{d}\tau\le \frac{1}{C}y(t),$$
  thus
  \begin{equation}
      \|u(\cdot,t)-\psi\|_{L^2}\le\left\|\int_{t}^{+\infty}u_t\,\mathrm{d}\tau\right\|_{L^2}\le\int_{t}^{+\infty}\|u_t\|_{L^2}\,\mathrm{d}\tau\le \frac{1}{C}y(t).\label{2.37}
  \end{equation}
  Combining (\ref{2.35}) and (\ref{2.36}) yields the differential inequality:
  \begin{equation}
      \frac{\,\mathrm{d}y}{\,\mathrm{d}t}+Cy^{\frac{1-\theta}{\theta}}\le 0.\label{2.38}
  \end{equation}
  \noindent
  When $\theta=\frac{1}{2}$, exponential decay follows directly from (\ref{2.37}) and (\ref{2.38}):
  $$\|u(\cdot,t)-\psi\|_{L^2}\le C_0e^{-C_1t}.$$
  \noindent
  When $\theta\in (0,\frac{1}{2})$, consider the equation:
  \begin{equation}
      \begin{cases}
          \frac{\,\mathrm{d}z}{\,\mathrm{d}t}+Cz^{\frac{1-\theta}{\theta}}=0 \\
          z(0)=y(0)>0 \label{2.39}
      \end{cases}
  \end{equation}
  The solution to (\ref{2.39}) is:
  $$z(t)=\left(-C\frac{2\theta-1}{\theta}t+y(0)^{\frac{2\theta-1}{\theta}}\right)^{\frac{\theta}{2\theta-1}}.$$
  Let $w(t)=z(t)-y(t)$. Subtracting (\ref{2.39}) from (\ref{2.38}):
  $$\frac{\,\mathrm{d}w}{\,\mathrm{d}t}+C(z^{\frac{1-\theta}{\theta}}-y^{\frac{1-\theta}{\theta}})\ge 0.$$
  Applying the mean value theorem:
  \begin{equation}
      \begin{cases}
          \frac{\,\mathrm{d}w}{\,\mathrm{d}t}+C\frac{1-\theta}{\theta}\xi^{\frac{1-2\theta}{\theta}}w\ge 0 \\
          w(0)=0 \label{2.40}
      \end{cases}
  \end{equation}
  where $\xi$ lies between $y$ and $z$.\\
  Multiplying (\ref{2.40}) by $e^{\int_{0}^{t}C\frac{1-\theta}{\theta}\xi(\tau)^{\frac{1-2\theta}{\theta}}\,\mathrm{d}\tau}$:
  $$\frac{\mathrm{d}}{\mathrm{d}t}\left(w(t)e^{\int_{0}^{t}C\frac{1-\theta}{\theta}\xi(\tau)^{\frac{1-2\theta}{\theta}}\,\mathrm{d}\tau}\right)\ge 0,$$
  and since $w(0)=0$:
  $$w(t)e^{\int_{0}^{t}C\frac{1-\theta}{\theta}\xi(\tau)^{\frac{1-2\theta}{\theta}}\,\mathrm{d}\tau}\ge 0,$$
  thus
  $$w(t)\ge 0,$$
  and therefore
  $$\|u(\cdot,t)-\psi\|_{L^2}\le \frac{1}{C}y(t)\le \frac{1}{C}z(t)=C(1+t)^{-\frac{\theta}{1-2\theta}}.$$

  Next, we prove that in $H^2\cap H_0^1(\Omega)$, (\ref{2.33}) and (\ref{2.34}) hold for $\theta=\frac{1}{2}$ and $\theta\in (0,\frac{1}{2})$ respectively.
  Let $v=u-\psi$. Then $v$ satisfies:
  \begin{equation}
    \begin{cases}
        v_{t}-\Delta v=\frac{\lambda}{(1-u)^2(1+\int_{\Omega}\frac{1}{1-u}\,\mathrm{d}x)^2}-\frac{\lambda}{(1-\psi)^2(1+\int_{\Omega}\frac{1}{1-\psi}\,\mathrm{d}x)^2},&x\in\Omega,t>0\\
        v=0,&x\in\partial\Omega,t>0\\
        v(x,0)=u_0(x)-\psi(x),&x\in\Omega
    \end{cases}\label{4.1}
\end{equation}
  Multiplying (\ref{4.1}) by $v$ and integrating over $\Omega$:
\begin{align*}
    &\quad\ \ \frac{1}{2}\frac{\mathrm{d}}{\mathrm{d}t}\|v\|_{L^2}^2+\|\nabla v\|_{L^2}^2\\
    &\ =\int_{\Omega}\left(\frac{\lambda}{(1-u)^2(1+\int_{\Omega}\frac{1}{1-u}\,\mathrm{d}x)^2}-\frac{\lambda}{(1-\psi)^2(1+\int_{\Omega}\frac{1}{1-\psi}\,\mathrm{d}x)^2}\right)v\,\mathrm{d}x,
\end{align*}
where
\begin{align*}
    &\quad\frac{\lambda}{(1-u)^2(1+\int_{\Omega}\frac{1}{1-u}\,\mathrm{d}x)^2}-\frac{\lambda}{(1-\psi)^2(1+\int_{\Omega}\frac{1}{1-\psi}\,\mathrm{d}x)^2}\\
    &\ =\lambda\left[\frac{1}{(1-u)^2(1+\int_{\Omega}\frac{1}{1-u}\,\mathrm{d}x)^2}-\frac{1}{(1-u)^2(1+\int_{\Omega}\frac{1}{1-\psi}\,\mathrm{d}x)^2}\right.\\
    &\quad\ \left.+\frac{1}{(1-u)^2(1+\int_{\Omega}\frac{1}{1-\psi}\,\mathrm{d}x)^2}-\frac{1}{(1-\psi)^2(1+\int_{\Omega}\frac{1}{1-\psi}\,\mathrm{d}x)^2}\right]\\
    &\ =\lambda\left[\frac{1}{(1-u)^2}\frac{(1+\int_{\Omega}\frac{1}{1-\psi}\,\mathrm{d}x)^2-(1+\int_{\Omega}\frac{1}{1-u}\,\mathrm{d}x)^2}{(1+\int_{\Omega}\frac{1}{1-u}\,\mathrm{d}x)^2(1+\int_{\Omega}\frac{1}{1-\psi}\,\mathrm{d}x)^2}\right.\\
    &\quad\ \left.+\frac{1}{(1+\int_{\Omega}\frac{1}{1-\psi}\,\mathrm{d}x)^2}\frac{(1-\psi)^2-(1-u)^2}{(1-u)^2(1-\psi)^2}\right]\\
    &\ =\lambda\left[\frac{1}{(1-u)^2}\frac{(2+\int_{\Omega}\frac{1}{1-u}\,\mathrm{d}x+\int_{\Omega}\frac{1}{1-\psi}\,\mathrm{d}x)(\int_{\Omega}\frac{-v}{(1-u)(1-\psi)}\,\mathrm{d}x)}{(1+\int_{\Omega}\frac{1}{1-u}\,\mathrm{d}x)^2(1+\int_{\Omega}\frac{1}{1-\psi}\,\mathrm{d}x)^2}\right.\\
    &\quad\ \left.+\frac{1}{(1+\int_{\Omega}\frac{1}{1-\psi}\,\mathrm{d}x)^2}\frac{(2-u-\psi)v}{(1-u)^2(1-\psi)^2}\right].
\end{align*}
Since $\psi\in C^{\infty}(\Omega)$ is a steady-state solution and $\|u\|_{L^{\infty}}\le 1-\delta$, we have $\|\frac{1}{1-\psi}\|_{L^{\infty}}\le M$ and $\|\frac{1}{1-u}\|_{L^{\infty}}\le M$ for all $t\ge 0$.\\
\noindent
Thus
\begin{align*}
    &\quad\frac{1}{2}\frac{\mathrm{d}}{\mathrm{d}t}\|v\|_{L^2}^2+\|\nabla v\|_{L^2}^2\\
    &\ =\lambda\left(\int_{\Omega}\frac{v}{(1-u)^2}\frac{(2+\int_{\Omega}\frac{1}{1-u}\,\mathrm{d}x+\int_{\Omega}\frac{1}{1-\psi}\,\mathrm{d}x)(\int_{\Omega}\frac{-v}{(1-u)(1-\psi)}\,\mathrm{d}x)}{(1+\int_{\Omega}\frac{1}{1-u}\,\mathrm{d}x)^2(1+\int_{\Omega}\frac{1}{1-\psi}\,\mathrm{d}x)^2}\,\mathrm{d}x\right.\\
    &\quad\ \left.+\int_{\Omega}\frac{1}{(1+\int_{\Omega}\frac{1}{1-\psi}\,\mathrm{d}x)^2}\frac{(2-u-\psi)v^2}{(1-u)^2(1-\psi)^2}\,\mathrm{d}x\right)\\
    &\ \le C\left(\int_{\Omega}\frac{|v|}{(1-u)^2}\,\mathrm{d}x\int_{\Omega}\frac{|-v|}{(1-u)(1-\psi)}\,\mathrm{d}x+\int_{\Omega}\frac{(2-u-\psi)v^2}{(1-u)^2(1-\psi)^2}\,\mathrm{d}x\right)\\
    &\ \le C\|v\|_{L^2}^{2},
\end{align*}
yielding
\begin{equation}
    \frac{\mathrm{d}}{\mathrm{d}t}\|v\|_{L^2}^2+2\|\nabla v\|_{L^2}^2\le C\|v\|_{L^2}^{2}.\label{4.2}
\end{equation}
Multiplying (\ref{4.1}) by $v_t$ and integrating over $\Omega$:
\begin{align*}
    &\quad\frac{1}{2}\frac{\mathrm{d}}{\mathrm{d}t}\|\nabla v\|_{L^2}^2+\|v_t\|_{L^2}^2\\
    &\ =\lambda\left(\int_{\Omega}\frac{v_t}{(1-u)^2}\frac{(2+\int_{\Omega}\frac{1}{1-u}\,\mathrm{d}x+\int_{\Omega}\frac{1}{1-\psi}\,\mathrm{d}x)(\int_{\Omega}\frac{-v}{(1-u)(1-\psi)}\,\mathrm{d}x)}{(1+\int_{\Omega}\frac{1}{1-u}\,\mathrm{d}x)^2(1+\int_{\Omega}\frac{1}{1-\psi}\,\mathrm{d}x)^2}\,\mathrm{d}x\right.\\
    &\quad\ \left.+\int_{\Omega}\frac{1}{(1+\int_{\Omega}\frac{1}{1-\psi}\,\mathrm{d}x)^2}\frac{(2-u-\psi)v_t v}{(1-u)^2(1-\psi)^2}\,\mathrm{d}x\right)\\
    &\ \le C\left(\int_{\Omega}\frac{|v_t|}{(1-u)^2}\,\mathrm{d}x\int_{\Omega}\frac{|-v|}{(1-u)(1-\psi)}\,\mathrm{d}x+\int_{\Omega}\frac{(2-u-\psi)v_t v}{(1-u)^2(1-\psi)^2}\,\mathrm{d}x\right)\\
    &\ \le C\|v\|_{L^2}\|v_t\|_{L^2}\\
    &\ \le \frac{(C\|v\|_{L^2})^2}{2}+\frac{\|v_t\|_{L^2}^2}{2}.
\end{align*}
Yielding
\begin{equation}
    \frac{\mathrm{d}}{\mathrm{d}t}\|\nabla v\|_{L^2}^2+\|v_t\|_{L^2}^2\le C\|v\|_{L^2}^{2}.\label{4.3}
\end{equation}
Differentiating (\ref{4.1}) with respect to $t$:
\begin{align*}
  &\quad v_{tt}-\Delta v_t\\
  &\ =\frac{2\lambda u_t}{(1-u)^3(1+\int_{\Omega}\frac{1}{1-u}\,\mathrm{d}x)^2}-\frac{2\lambda\int_{\Omega}\frac{u_t}{(1-u)^2}\,\mathrm{d}x}{(1-u)^2(1+\int_{\Omega}\frac{1}{1-u}\,\mathrm{d}x)^3}-\frac{2\lambda \psi_t}{(1-\psi)^3(1+\int_{\Omega}\frac{1}{1-\psi}\,\mathrm{d}x)^2}\\
  &\quad\ +\frac{2\lambda\int_{\Omega}\frac{\psi_t}{(1-\psi)^2}\,\mathrm{d}x}{(1-\psi)^2(1+\int_{\Omega}\frac{1}{1-\psi}\,\mathrm{d}x)^3}.
\end{align*}
Since $\psi$ is a steady-state solution, $\psi_t=0$ and $v_t=u_t$, so
\begin{equation}
    v_{tt}-\Delta v_t=\frac{2\lambda v_t}{(1-u)^3(1+\int_{\Omega}\frac{1}{1-u}\,\mathrm{d}x)^2}-\frac{2\lambda\int_{\Omega}\frac{v_t}{(1-u)^2}\,\mathrm{d}x}{(1-u)^2(1+\int_{\Omega}\frac{1}{1-u}\,\mathrm{d}x)^3}.\label{4.4}
\end{equation}
Multiplying (\ref{4.4}) by $v_t$ and integrating over $\Omega$:
\begin{align*}
    &\quad \frac{1}{2}\frac{\mathrm{d}}{\mathrm{d}t}\|v_t\|_{L^2}^2+\|\nabla v_t\|_{L^2}^2\\
    &\ =\frac{2\lambda}{(1+\int_{\Omega}\frac{1}{1-u}\,\mathrm{d}x)^2}\int_{\Omega}\frac{v_{t}^{2}}{(1-u)^3}\,\mathrm{d}x-\frac{2\lambda(\int_{\Omega}\frac{v_t}{(1-u)^2}\,\mathrm{d}x)^2}{(1+\int_{\Omega}\frac{1}{1-u}\,\mathrm{d}x)^3}\\
    &\le \frac{2\lambda}{(1+\int_{\Omega}\frac{1}{1-u}\,\mathrm{d}x)^2}\int_{\Omega}\frac{v_{t}^{2}}{(1-u)^3}\,\mathrm{d}x\\
    &\le C_2\|v_t\|_{L^2}^{2},
\end{align*}
yielding
\begin{equation}
    \frac{1}{4C_2}\frac{\mathrm{d}}{\mathrm{d}t}\|v_t\|_{L^2}^2\le \frac{1}{2}\|v_t\|_{L^2}^{2}.\label{4.5}
\end{equation}
Adding (\ref{4.2}), (\ref{4.3}) and (\ref{4.5}):
$$\frac{\mathrm{d}}{\mathrm{d}t}\left(\|v\|_{L^2}^2+\|\nabla v\|_{L^2}^2+\frac{1}{4C_2}\|v_t\|_{L^2}^2\right)+2\|\nabla v\|_{L^2}^2+\frac{1}{2}\|v_t\|_{L^2}^2\le C\|v\|_{L^2}^{2},$$
thus
$$\frac{\mathrm{d}}{\mathrm{d}t}\left(\|v\|_{L^2}^2+\|\nabla v\|_{L^2}^2+\frac{1}{4C_2}\|v_t\|_{L^2}^2\right)+\|v\|_{L^2}^{2}+2\|\nabla v\|_{L^2}^2+\frac{1}{2}\|v_t\|_{L^2}^2\le (C+1)\|v\|_{L^2}^{2}.$$
Note there exists $\alpha>0$ such that
$$\|v\|_{L^2}^{2}+2\|\nabla v\|_{L^2}^2+\frac{1}{2}\|v_t\|_{L^2}^2\ge \alpha\left(\|v\|_{L^2}^2+\|\nabla v\|_{L^2}^2+\frac{1}{4C_2}\|v_t\|_{L^2}^2\right),$$
therefore
$$\frac{\mathrm{d}}{\mathrm{d}t}\widetilde{y}(t)+\alpha\widetilde{y}(t)\le C\|v\|_{L^2}^{2},$$
where $\widetilde{y}(t)=\|v\|_{L^2}^2+\|\nabla v\|_{L^2}^2+\frac{1}{4C_2}\|v_t\|_{L^2}^2$, yielding
\begin{equation}
    \frac{\,\mathrm{d}\widetilde{y}}{\,\mathrm{d}t}+\alpha\widetilde{y}\le C\|v\|_{L^2}^{2}.\label{4.6}
\end{equation}
When $\theta\in (0,\frac{1}{2})$, we have already shown $\|v\|_{L^2}\le C(1+t)^{-\frac{\theta}{1-2\theta}}$, so
$$\frac{\,\mathrm{d}\widetilde{y}}{\,\mathrm{d}t}+\alpha\widetilde{y}\le C(1+t)^{-\frac{2\theta}{1-2\theta}}.$$
Thus
\begin{align*}
    \widetilde{y}(t)&\le Ce^{-\alpha t}+Ce^{-\alpha t}\int_{0}^{t}(1+\tau)^{-\frac{2\theta}{1-2\theta}}e^{\alpha \tau}\,\mathrm{d}\tau\\
    &\le Ce^{-\alpha t}+Ce^{-\alpha t}\left[\int_{0}^{\frac{t}{2}}(1+\tau)^{-\frac{2\theta}{1-2\theta}}\,\mathrm{d}\tau\ e^{\frac{\alpha t}{2}}+(1+\frac{t}{2})^{-\frac{2\theta}{1-2\theta}}\int_{\frac{t}{2}}^{t}e^{\alpha\tau}\,\mathrm{d}\tau\right]\\
    &\le Ce^{-\alpha t}+Ce^{-\frac{\alpha t}{2}}\int_{0}^{\frac{t}{2}}(1+\tau)^{-\frac{2\theta}{1-2\theta}}\,\mathrm{d}\tau +Ce^{-\alpha t}(1+\frac{t}{2})^{-\frac{2\theta}{1-2\theta}}\left[\frac{1}{\alpha}(e^{\alpha t}-e^{\frac{\alpha t}{2}})\right]\\
    &\le C(1+t)^{-\frac{2\theta}{1-2\theta}}.
\end{align*}
Hence
$$\|\nabla v\|_{L^2}\le C(1+t)^{-\frac{\theta}{1-2\theta}},\quad \|v_t\|_{L^2}\le C(1+t)^{-\frac{\theta}{1-2\theta}},$$
and therefore
\begin{align*}
    \|v\|_{H^2}&\le C\|\Delta v\|_{L^2}\\
    &\le C\left(\|v_t\|_{L^2}+\left\|\frac{\lambda}{(1-u)^2(1+\int_{\Omega}\frac{1}{1-u}\,\mathrm{d}x)^2}-\frac{\lambda}{(1-\psi)^2(1+\int_{\Omega}\frac{1}{1-\psi}\,\mathrm{d}x)^2}\right\|_{L^2}\right)\\
    &\le C\|v_t\|_{L^2}+C\left\|\frac{1}{(1-u)^2}\int_{\Omega}\frac{-v}{(1-u)(1-\psi)}\,\mathrm{d}x\right\|_{L^2}+C\left\|\frac{(2-u-\psi)v}{(1-u)^2(1-\psi)^2}\right\|_{L^2}\\
    &\le C\|v_t\|_{L^2}+C\|v\|_{L^2}\\
    &\le C(1+t)^{-\frac{\theta}{1-2\theta}}.
\end{align*}
When $\theta=\frac{1}{2}$, we have $\|v\|_{L^2}\le C_0 e^{-C_1 t}$, so
$$\frac{\,\mathrm{d}\widetilde{y}}{\,\mathrm{d}t}+\alpha\widetilde{y}\le Ce^{-2C_1 t}.$$
Thus
\begin{align*}
    \widetilde{y}(t)&\le Ce^{-\alpha t}+Ce^{-\alpha t}\int_{0}^{t}e^{(\alpha-2C_1)\tau}\,\mathrm{d}\tau\\
    &= Ce^{-\alpha t}+Ce^{-\alpha t}\frac{1}{\alpha-2C_1}\left(e^{(\alpha-2C_1)t}-1\right)\\
    &\le Ce^{-\alpha t}+Ce^{-2C_1 t}+Ce^{-\alpha t}\\
    &\le Ce^{-\min\{\alpha,2C_1\}t}.
\end{align*}
Hence
$$\|\nabla v\|_{L^2}\le Ce^{-\min\{\frac{\alpha}{2},C_1\}t},\quad \|v_t\|_{L^2}\le Ce^{-\min\{\frac{\alpha}{2},C_1\}t},$$
and similarly
$$\|v\|_{H^2}\le C\|\Delta v\|_{L^2}\le C\|v_t\|_{L^2}+C\|v\|_{L^2}\le Ce^{-\min\{\frac{\alpha}{2},C_1\}t}.$$
The theorem is proved.
\end{proof}
	
\section{Numerical experiments}

\qquad In the reference \cite{ref5}, N. I. Kavallaris et al. plotted solutions of the equation in one dimension and radially symmetric solutions on a circular domain in two dimensions. In this section, besides reproducing these plots, we will also visualize general solutions on both circular and square domains in 2D. We will observe the asymptotic behavior of solutions, focusing on their evolution with respect to $\lambda$, and investigate whether solutions exhibit dichotomy with respect to the voltage $\lambda$.

\subsection{One-Dimensional Case}
\qquad We first study the numerical solution of the nonlocal parabolic MEMS equation in one-dimensional space.

Consider the equation:
\begin{equation}
    \begin{cases}
        u_{t}=u_{xx}+\frac{\lambda}{(1-u)^2(1+\int_{0}^{1}\frac{1}{1-u}\,\mathrm{d}x)^2},&x\in (0,1),t>0\\
        u(0,t)=0,\quad u(1,t)=0,&t>0\\
        u(x,0)=0,&x\in(0,1)
    \end{cases}\label{3.1}
\end{equation}

Reference \cite{ref5} reports a critical voltage $\lambda^{*}\approx 8.533$. We compare solutions for $\lambda=8.53$ and $\lambda=8.54$.

\begin{figure}[H]
    \centering
    \begin{minipage}[b]{0.45\textwidth}
      \includegraphics[width=\textwidth]{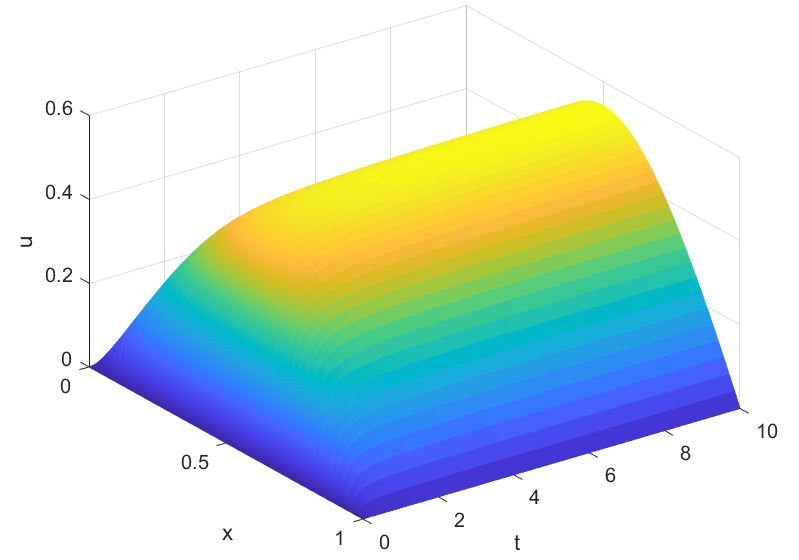}
      \caption{Solution of equation (\ref{3.1}) for $\lambda=8.53$ in $t\in[0,10]$}
      \label{fig:image1}
    \end{minipage}
    \hfill
    \begin{minipage}[b]{0.45\textwidth}
      \includegraphics[width=\textwidth]{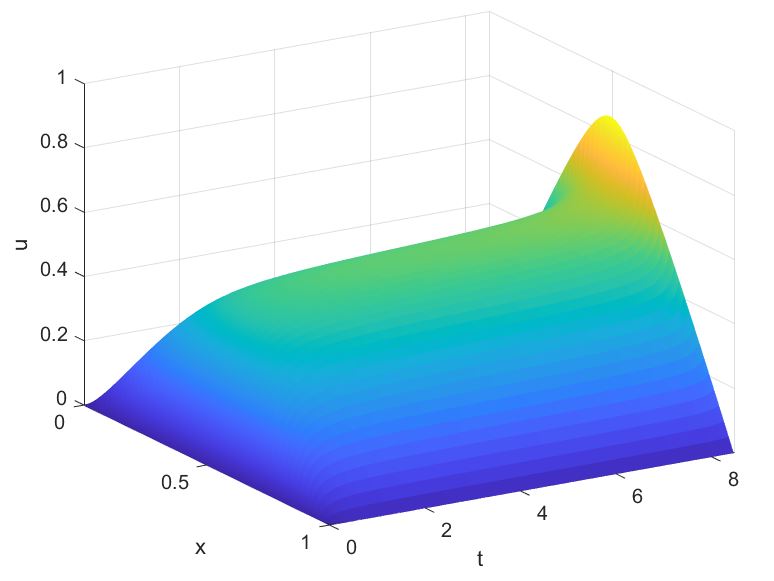}
      \caption{Solution of equation (\ref{3.1}) for $\lambda=8.54$ in $t\in[0,8.5]$}
      \label{fig:image2}
    \end{minipage}
  \end{figure}

Figure \ref{fig:image1} and \ref{fig:image2} show that when $\lambda=8.53$, the solution converges monotonically as $t\to +\infty$, while for $\lambda=8.54$, the solution quenches in finite time.

The maximum solution value occurs near $x=0.5$. Fixing $x=0.5$, we plot the solution value at this point over time for different $\lambda$:

\begin{figure}[H]
    \centering
    \includegraphics[width=10cm,height=6cm]{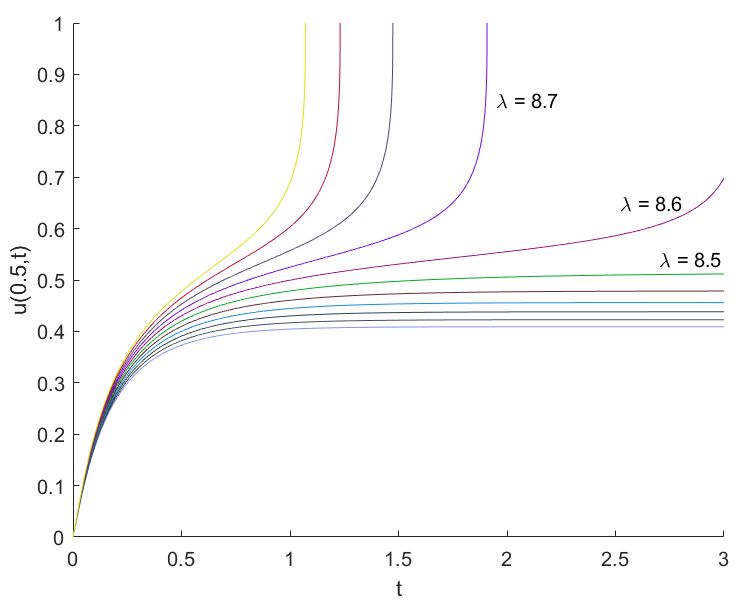}
    \caption{Solution values at $x=0.5$ for $\lambda=8$ to $9$ in equation (\ref{3.1})}\label{fig:image3}
\end{figure}

Figure \ref{fig:image3} reveals dichotomy with respect to $\lambda$ at $x=0.5$. The critical voltage $\lambda^{*}$ lies between 8.5 and 8.6, consistent with reference \cite{ref1}.

Fixing $t=10$, we plot solutions of equation (\ref{3.1}) for different $\lambda(<\lambda^{*})$:

\begin{figure}[H]
    \centering
    \includegraphics[width=10cm,height=6cm]{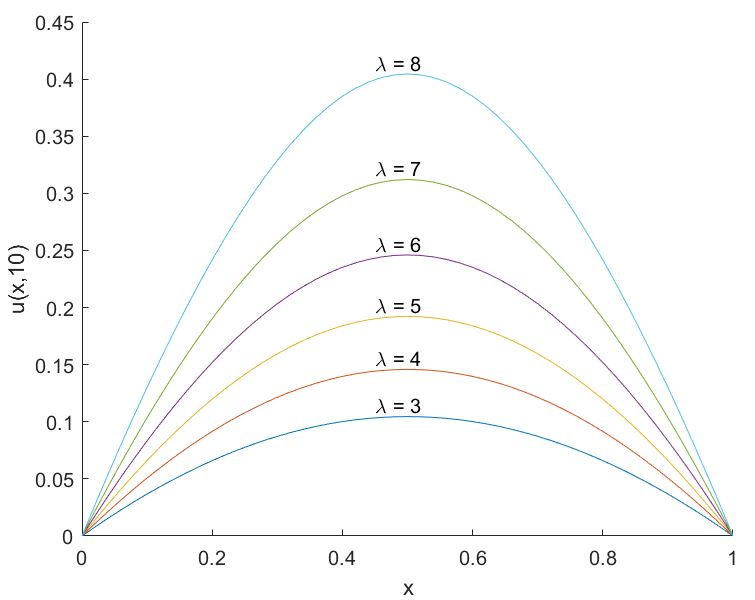}
    \caption{Solutions of equation (\ref{3.1}) at $t=10$ for $\lambda=3$ to $8$}\label{fig:image12}
\end{figure}

For $\lambda<\lambda^{*}$, the solution $u(\cdot,10)$ increases monotonically with $\lambda$.

\subsection{Two-Dimensional Case}

\subsubsection{Circular Domain}
\qquad Let $\Omega=B(0,1)$. First, we study radially symmetric solutions of equation (\ref{1.1}) in 2D, which simplifies:
\begin{equation}
    \begin{cases}
        u_{t}=u_{rr}+\frac{1}{r}u_r+\frac{\lambda}{(1-u)^2(1+2\pi\int_{0}^{1}\frac{r}{1-u}\,\mathrm{d}r)^2},\quad 0<r<1,\ t>0\\
        u_r(0,t)=0,\quad u(1,t)=0,\\
        u(r,0)=0.
    \end{cases}\label{3.3}
\end{equation}
Solutions for $\lambda=22$ and $\lambda=22.5$ are plotted below:

\begin{figure}[H]
    \centering
    \begin{minipage}[b]{0.45\textwidth}
      \includegraphics[width=\textwidth]{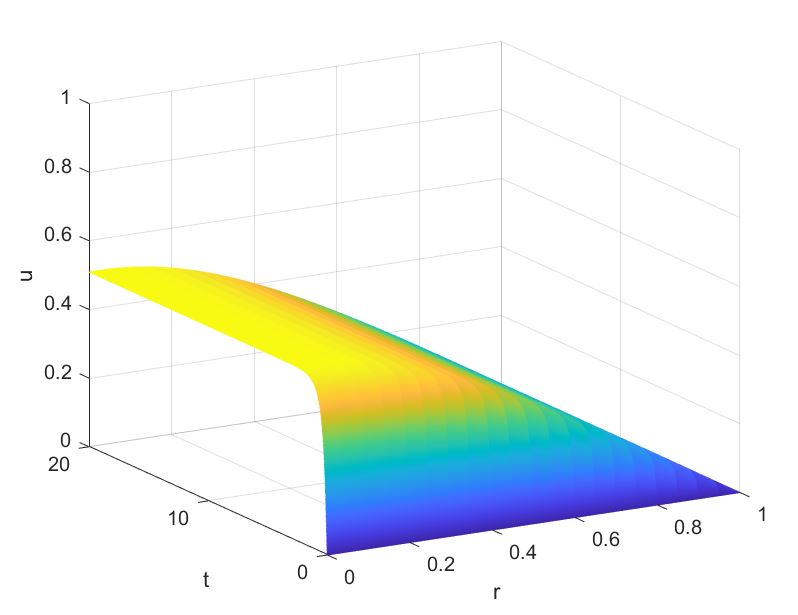}
      \caption{Solution of equation (\ref{3.3}) for $\lambda=22$ in $r\in [0,1],\ t\in [0,20]$}\label{fig:image13}
    \end{minipage}
    \hfill
    \begin{minipage}[b]{0.45\textwidth}
      \includegraphics[width=\textwidth]{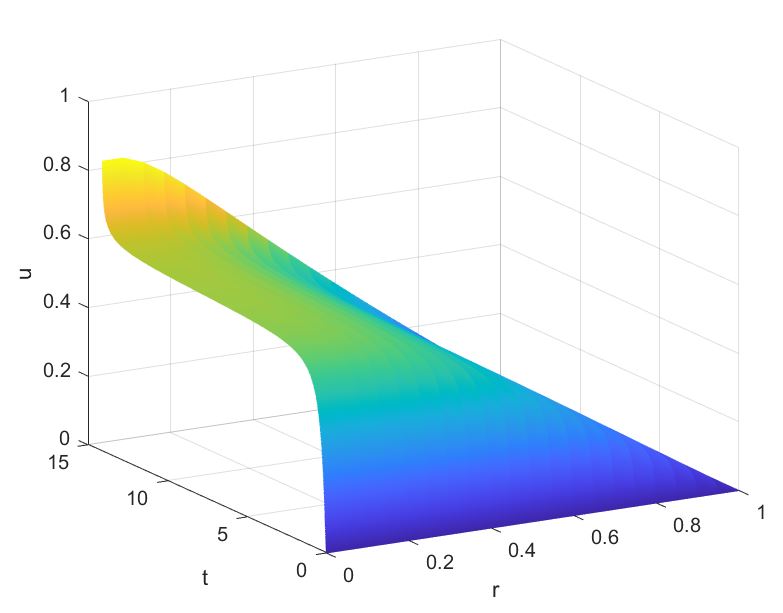}
      \caption{Solution of equation (\ref{3.3}) for $\lambda=22.5$ in $r\in [0,1],\ t\in [0,14]$}\label{fig:image8}
    \end{minipage}
  \end{figure}

Figure \ref{fig:image13} shows convergence as $t\to +\infty$ for $\lambda=22$, while Figure \ref{fig:image8} shows finite-time quenching for $\lambda=22.5$.

The maximum occurs at $r=0$. Fixing $r=0$, we plot solution values at this point:

\begin{figure}[H]
    \centering
    \includegraphics[width=10cm,height=6cm]{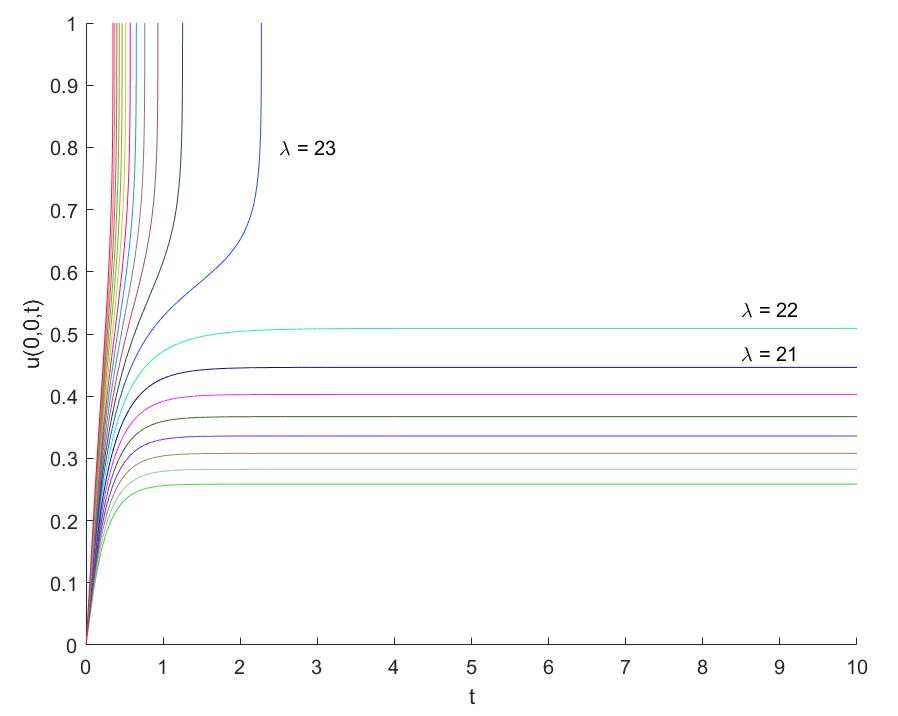}
    \caption{Solution values at $r=0$ for $\lambda=15$ to $34$ in equation (\ref{3.3})}\label{fig:image10}
\end{figure}

Figure \ref{fig:image10} shows dichotomy with respect to $\lambda$ at $r=0$, suggesting $\lambda^{*}\in(22,23)$.

Fixing $t=10$, we plot solutions for $\lambda(<\lambda^{*})$:

\begin{figure}[H]
    \centering
    \includegraphics[width=10cm,height=6cm]{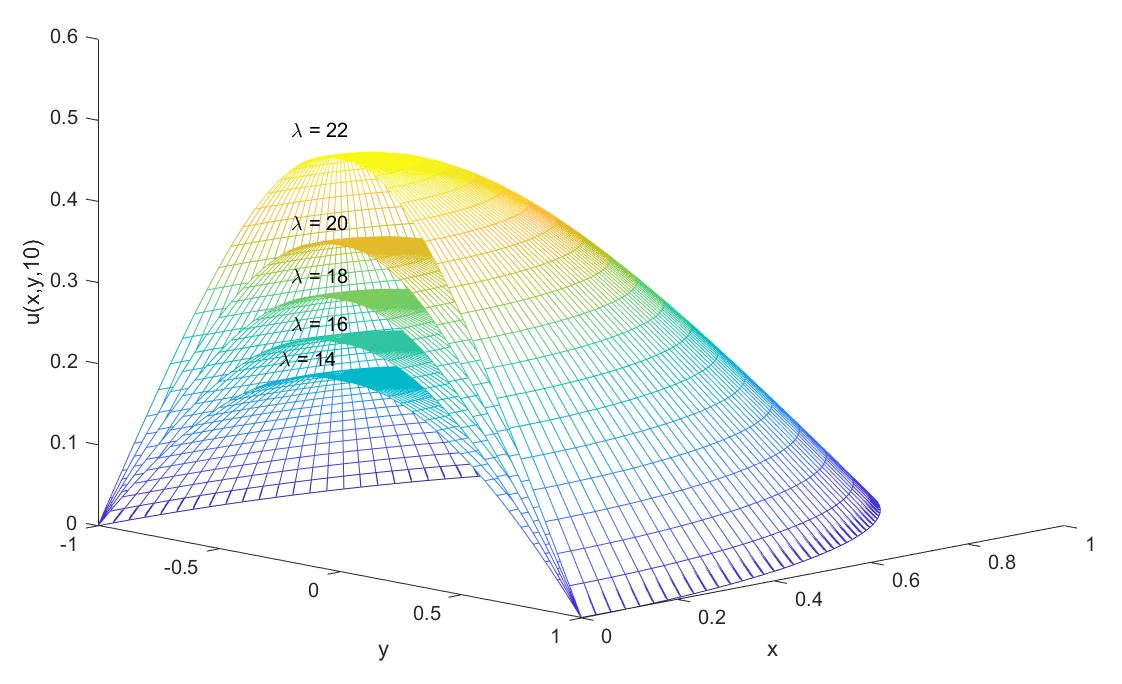}
    \caption{Solutions of equation (\ref{3.3}) at $t=10$ for $\lambda=14$ to $22$}\label{fig:image14}
\end{figure}

For $\lambda<\lambda^{*}$, $u(\cdot,10)$ increases monotonically with $\lambda$.

Next, we consider non-radially symmetric solutions of equation (\ref{1.1}):
\begin{equation}
    \begin{cases}
        u_{t}=\Delta u + \frac{\lambda}{(1-u)^2(1+\int_{B(0,1)}\frac{1}{1-u}\,\mathrm{d}x\,\mathrm{d}y)^2},&(x,y)\in B(0,1),t>0\\
        u=0,&(x,y)\in\partial B(0,1),t>0\\
        u(x,y,0)=u_0(x,y),&(x,y)\in B(0,1)
    \end{cases}\label{3.4}
\end{equation}
We choose a non-radially symmetric initial value in $H^2\cap H_0^1(\Omega)$:
$$u_0(x,y)=100\ (1-x^2-y^2)^3\ x^2\ y^2.$$
Solutions for $\lambda=20$ are shown below:

\begin{figure}[H]
    \centering
    \includegraphics[width=15cm,height=5cm]{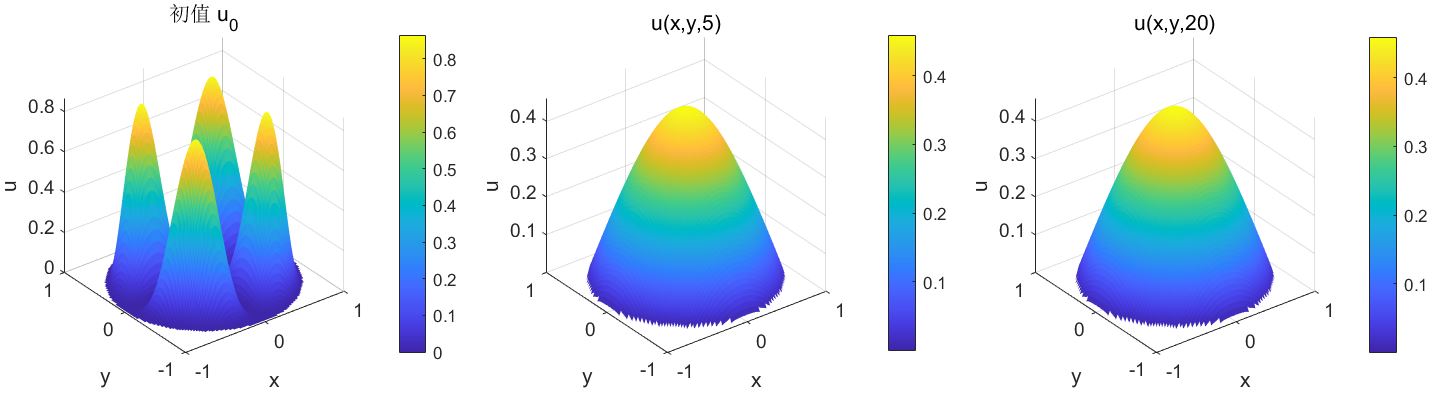}
    \caption{Initial value $u_0$ and solutions at $t=5,20$ for $\lambda=20$ in equation (\ref{3.4})}\label{fig:image18}
\end{figure}

Figure \ref{fig:image18} confirms global existence and convergence as $t\to\infty$ for $\lambda=20$. Solutions for different $\lambda$:

\begin{figure}[H]
    \centering
    \includegraphics[width=15cm,height=5cm]{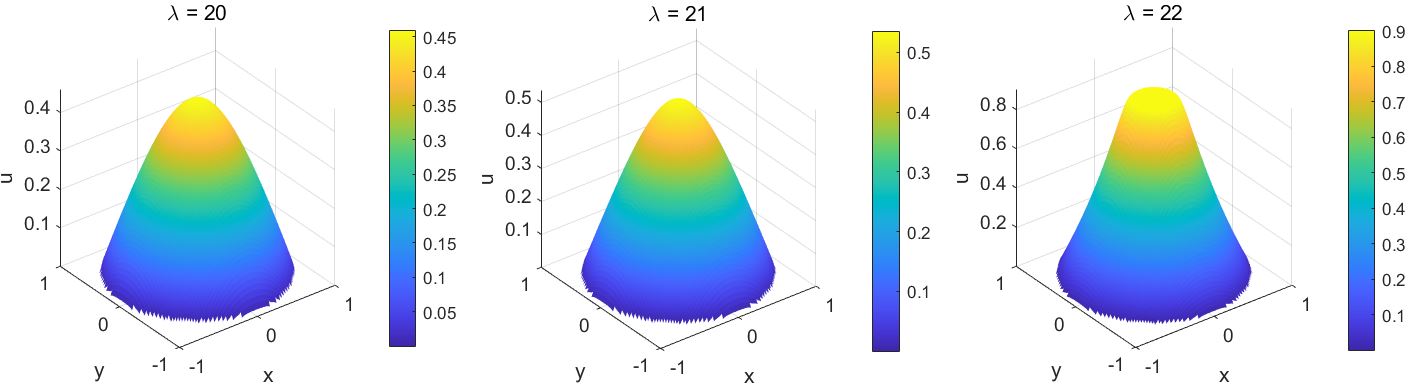}
    \caption{Solutions $u(x,y,5)$ for $\lambda=20$ to $22$ in equation (\ref{3.4})}\label{fig:image16}
\end{figure}

\begin{figure}[H]
    \centering
    \includegraphics[width=10cm,height=6cm]{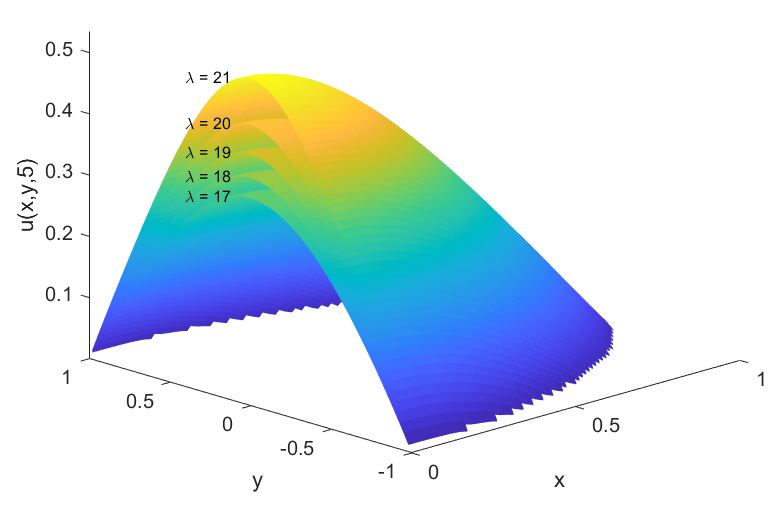}
    \caption{Comparison of $u(x,y,5)$ for $\lambda=17$ to $21$}\label{fig:image17}
\end{figure}

Figure \ref{fig:image16} and \ref{fig:image17} show that $u(\cdot,5)$ increases with $\lambda$, and quenching occurs when $\lambda$ exceeds a critical value.

\subsubsection{Square Domain}
\qquad Let $\Omega=[0,1]\times [0,1]$. We study solutions of equation (\ref{1.1}):
\begin{equation}
    \begin{cases}
        u_{t}=\Delta u + \frac{\lambda}{(1-u)^2(1+\iint\limits_{\Omega}\frac{1}{1-u}\,\mathrm{d}x\,\mathrm{d}y)^2},&(x,y)\in\Omega,t>0\\
        u=0,&(x,y)\in\partial\Omega,t>0\\
        u(x,y,0)=0,&(x,y)\in\Omega
    \end{cases}\label{3.2}
\end{equation}

Solutions for $\lambda=10$ and $\lambda=16$:

\begin{figure}[H]
    \centering
    \includegraphics[width=13cm,height=5cm]{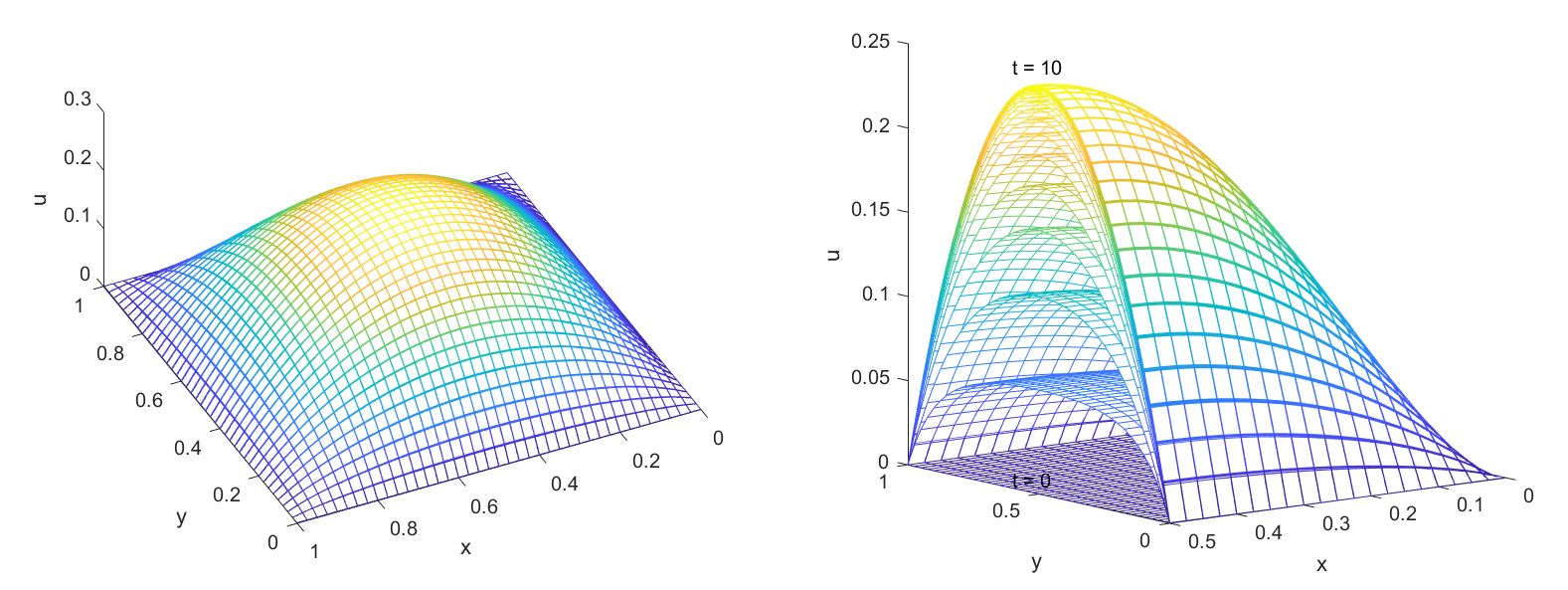}
    \caption{Solution of equation (\ref{3.2}) for $\lambda=10$ in $t\in [0,10]$}\label{fig:image4}
\end{figure}

\begin{figure}[H]
    \centering
    \includegraphics[width=10cm,height=6cm]{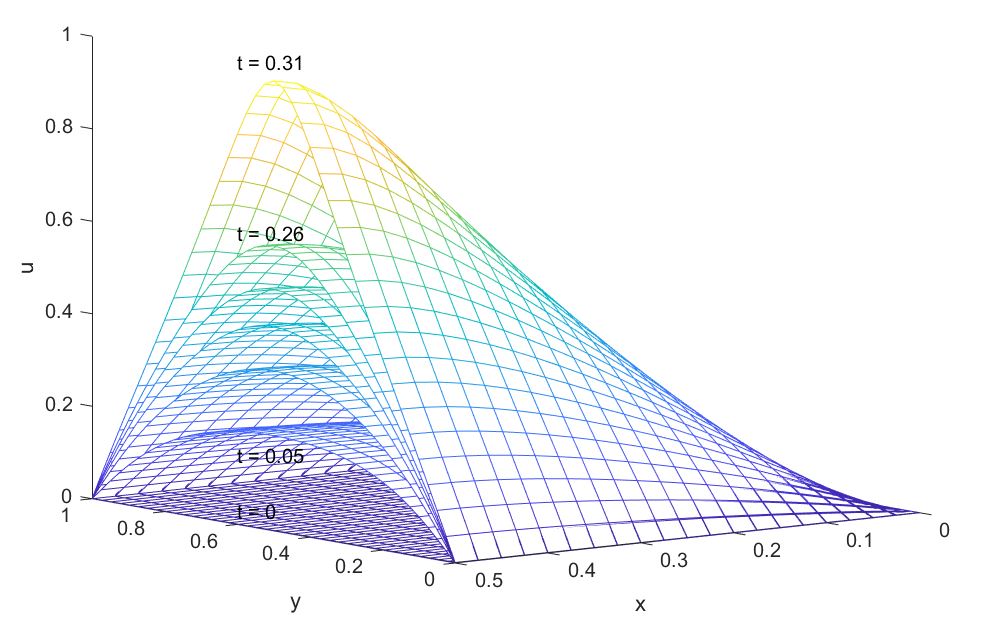}
    \caption{Solution of equation (\ref{3.2}) for $\lambda=16$ in $t\in [0,0.31]$}\label{fig:image5}
\end{figure}

Figure \ref{fig:image4} shows convergence for $\lambda=10$, while Figure \ref{fig:image5} shows quenching for $\lambda=16$. Asymptotic behavior for different $\lambda$:

\begin{figure}[H]
    \centering
    \includegraphics[width=10cm,height=6cm]{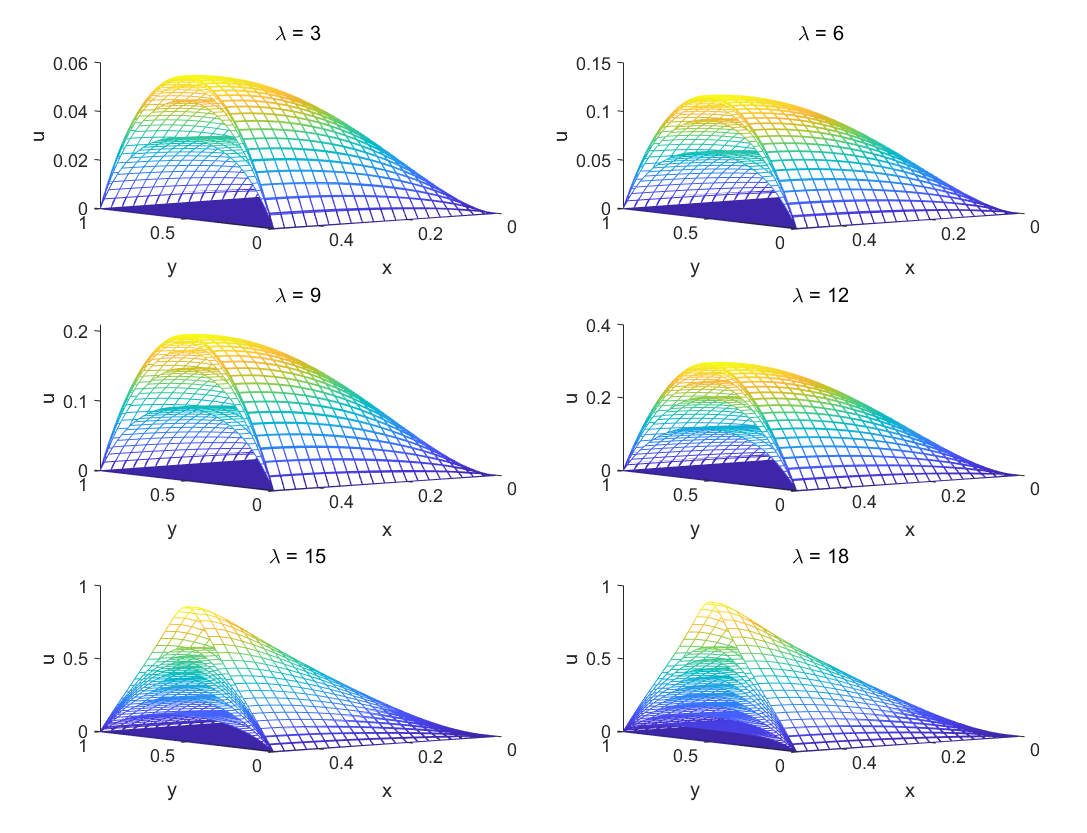}
    \caption{Solution evolution for different $\lambda$ in equation (\ref{3.2})}\label{fig:image7}
\end{figure}

Figure \ref{fig:image7} confirms convergence for small $\lambda$ and finite-time quenching for large $\lambda$.

The maximum occurs near $(0.5,0.5)$. Fixing this point, we plot solution values:

\begin{figure}[H]
    \centering
    \includegraphics[width=10cm,height=6cm]{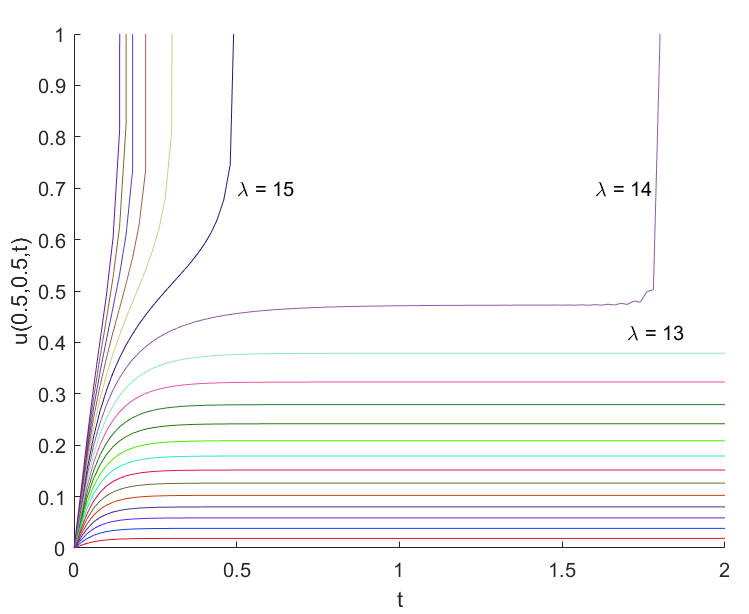}
    \caption{Solution values at $(0.5,0.5)$ for $\lambda=1$ to $20$ in equation (\ref{3.2})}\label{fig:image6}
\end{figure}

Figure \ref{fig:image6} shows dichotomy with respect to $\lambda$ at $(0.5,0.5)$, suggesting $\lambda^{*}\in(13,14)$.

Fixing $t=1$, we plot solutions for $\lambda(<\lambda^{*})$:

\begin{figure}[H]
    \centering
    \includegraphics[width=10cm,height=6cm]{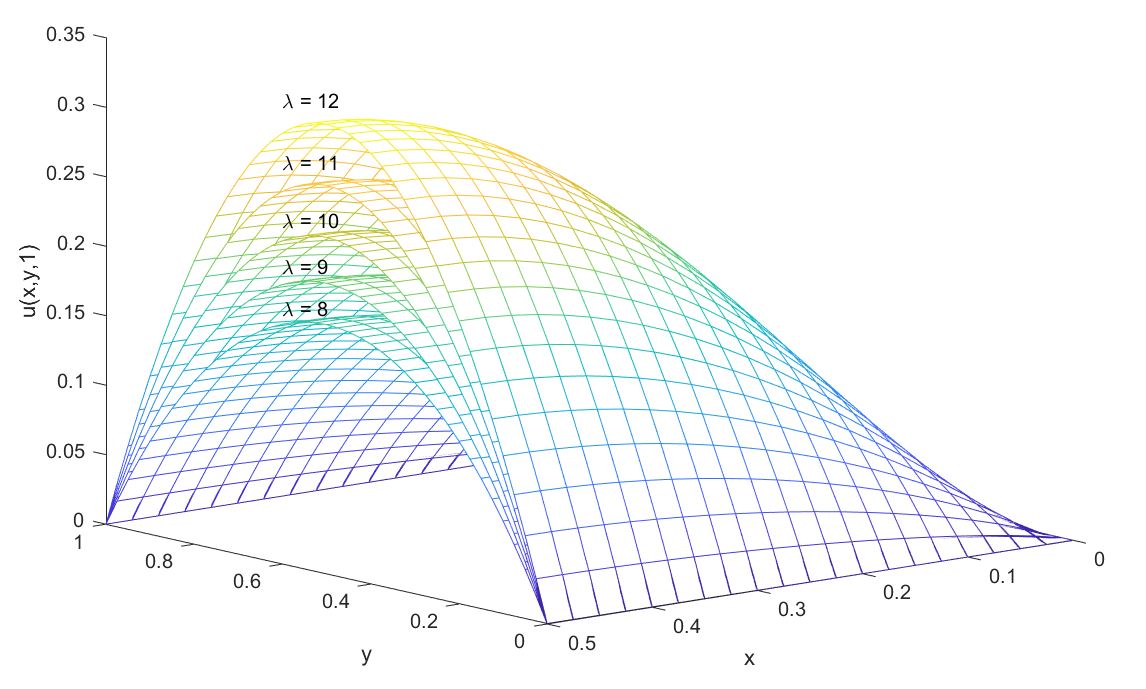}
    \caption{Solutions of equation (\ref{3.2}) at $t=1$ for $\lambda=8$ to $12$}\label{fig:image15}
\end{figure}

For $\lambda<\lambda^{*}$, $u(\cdot,1)$ increases monotonically with $\lambda$.

\subsection{Conjectures}

Based on these numerical experiments, we propose:

\begin{conjecture}\label{conjecture3.1}
    Let $\Omega\subset \mathbb{R}^N(1\le N\le 3)$ be a bounded domain with smooth boundary. For some $u_0\in H^2\cap H_0^1(\Omega)$, there exists $\lambda^{*}>0$ such that:
    \begin{itemize}
        \item For $\lambda<\lambda^{*}$, equation (\ref{1.1}) has a unique global solution $u_{\lambda}$ converging to a steady state as $t\to \infty$.
        \item For $\lambda>\lambda^{*}$, the solution quenches in finite time and the steady-state equation (\ref{1.4}) has no solution.
    \end{itemize}
\end{conjecture}

\begin{conjecture}
    For $\lambda<\lambda^{*}$, let $u_{\lambda}$ be defined as in Conjecture \ref{conjecture3.1}. Then $\forall\ t<\infty$, the mapping $\lambda\longmapsto u_{\lambda}(\cdot,t)$ is monotonically increasing in $H^2\cap H_0^1(\Omega)$.
\end{conjecture}

\end{document}